\documentclass[12pt,reqno]{amsart}

\usepackage{mathtools,amssymb,amsthm}
\usepackage{esint}
\usepackage[a4paper,margin=2.5cm]{geometry}
\usepackage[english]{babel}
\usepackage[backref=page]{hyperref}
\usepackage{orcidlink}

\allowdisplaybreaks[1]
\newtheorem{theorem}{Theorem}[section]
\newtheorem{lemma}[theorem]{Lemma}
\newtheorem{proposition}[theorem]{Proposition}
\theoremstyle{remark}
\newtheorem*{remark}{Remark}
\numberwithin{equation}{section}

\DeclareMathOperator{\Div}{Div}
\DeclareMathOperator{\tr}{tr}
\DeclareMathOperator{\BMO}{BMO}

\newcommand{\RR}{\mathbb R}
\newcommand{\loc}{\mathrm{loc}}
\newcommand{\supp}{\operatorname{supp}}
\newcommand{\ol}[1]{\overline{#1}}

\title[Interior Regularity for Stable Solutions]{Interior Regularity for Stable Solutions of Semilinear Elliptic Equations with Sign-Changing Nonlinearities}

\author[Long-Han Huang]{Long-Han Huang\orcidlink{0009-0003-1864-2829}}
\address{Department of Mathematical Sciences, Tsinghua University, Beijing 100084, China}
\email{\href{mailto:huanglh22@mails.tsinghua.edu.cn}{\texttt{huanglh22@mails.tsinghua.edu.cn}}}

\author{Yuxuan Zhou}
\address{Department of Mathematical Sciences, Tsinghua University, Beijing 100084, China}
\email{\href{mailto:zhou-yx22@mails.tsinghua.edu.cn}{\texttt{zhou-yx22@mails.tsinghua.edu.cn}}}

\author[Wenming Zou]{Wenming Zou\orcidlink{0000-0002-4178-9568}}
\address{Department of Mathematical Sciences, Tsinghua University, Beijing 100084, China}
\email{\href{mailto:zou-wm@mail.tsinghua.edu.cn}{\texttt{zou-wm@mail.tsinghua.edu.cn}}}

\date{\today}

\begin{document}

\begin{abstract}
We investigate interior regularity for stable solutions of the semilinear elliptic equation $-\Delta u=f(u)$ in an open set $\Omega\subseteq \mathbb R^n$, where $f$ is locally Lipschitz and may change sign. We establish sharp interior H\"older estimates for stable classical solutions in every dimension $n\le 9$, thereby fully resolving a question raised by Cabr\'e \cite{Cabre2023,Cabre2026} concerning interior regularity without the nonnegativity assumption on $f$. For $n\le 5$, we also give a short alternative proof of the corresponding H\"older regularity result due to Peng, Zhang, and Zhou \cite{PengZhangZhou2026}. We further establish higher-integrability estimates for the gradient and $L^1$ bounds for the Hessian in all dimensions. For $n\ge 10$, we obtain BMO and Morrey estimates, extending related results previously established for nonnegative nonlinearities. When $f$ is convex, we extend the interior regularity estimates to stable weak solutions through an approximation theorem. As applications, we prove Liouville-type theorems for entire stable solutions of equations with sign-changing nonlinearities across all dimensions.
\end{abstract}

\maketitle
\setcounter{tocdepth}{2}
\tableofcontents

\section{Introduction}

\subsection{Background and motivation}

Let \(n \in \mathbb{N}^{+}\) and let \(\Omega \subseteq \mathbb{R}^{n}\) be an open set. We study the semilinear elliptic equation
\begin{equation}\label{eq:pde}
 -\Delta u=f(u)\qquad\hbox{in }\Omega,
\end{equation}
with $f \in C_{\loc}^{0,1}(\RR)$, i.e. $f$ is a locally Lipschitz function in $\RR$. We say that $u\in H^1_{\loc}(\Omega)$ is a weak solution of \eqref{eq:pde} if
$f(u)\in L^1_{\loc}(\Omega)$ and
\[
 \int_\Omega Du\cdot D\varphi=\int_\Omega f(u)\varphi,
 \qquad \forall\,\varphi\in C_c^\infty(\Omega).
\]
For $t\in\RR$, set
\[
 f'_-(t)=\liminf_{h\to 0}\frac{f(t+h)-f(t)}{h},
 \qquad f^-(t)=\max\{-f(t),0\}.
\]
A weak solution $u$ is called stable if
$f'_-(u)\in L^1_{\loc}(\Omega)$ and
\begin{equation}\label{eq:stability}
 Q_u(\xi):=\int_\Omega\bigl(|D\xi|^2-f'_-(u)\xi^2\bigr)\ge0,
 \qquad \forall\,\xi\in C_c^\infty(\Omega).
\end{equation}
We also use $Q_u$ to denote the associated symmetric bilinear form.

Our focus is the regularity of stable solutions to \eqref{eq:pde}, a problem with a long history in elliptic theory. Crandall and Rabinowitz initiated the study of this problem in \cite{CrandallRabinowitz1975}. The topic has since received considerable attention. For a comprehensive account and a detailed presentation of known results, we direct the reader to the monograph \cite{Dupaigne2011} and to the surveys \cite{Cabre2017,Cabre2026}. For related developments on the regularity of stable minimal hypersurfaces, considered as a geometric analog, see \cite{Simons1968,BombieriDeGiorgiGiusti1969,SchoenSimonYau1975,FischerColbrieSchoen1980,ChodoshLi2024,ChodoshLiMinterStryker2026,Mazet2024,HongLiWang2026}.

The regularity theory for stable solutions to \eqref{eq:pde} is closely tied to the classical Gelfand problem \cite{Gelfand1963}:
	\begin{align}\label{eq:gelfand}
		\left\{
		\begin{aligned}
			-\Delta u &= \lambda f(u) && \text{in }\Omega,\\
			u &> 0 && \text{in }\Omega,\\
			u &= 0 && \text{on }\partial\Omega.
		\end{aligned}
		\right.
	\end{align}
Here $\Omega \subseteq \RR^n$ is a bounded smooth connected open set, $\lambda > 0$, and $f: [0,+\infty) \to \RR$ satisfies $f(0)>0$ and is nondecreasing, convex, and superlinear at $+\infty$, meaning $\lim\limits_{t \to +\infty} \frac{f(t)}{t} = +\infty$. It is known \cite{BrezisVazquez1997, Martel1997, Brezis2003, Dupaigne2011} that there exists $\lambda^{*} = \lambda^{*}(\Omega,f) \in (0,+\infty)$ such that for all $\lambda \in (0,\lambda^{*})$ the problem \eqref{eq:gelfand} admits a stable classical solution $u_{\lambda} \in C^{2}(\Omega) \cap C(\ol{\Omega})$, whereas no classical solution exists for $\lambda > \lambda^{*}$. At $\lambda = \lambda^{*}$, the problem \eqref{eq:gelfand} possesses a unique $L^{1}$ solution $u^{*}$, obtained  as the pointwise limit of the increasing family $\{u_{\lambda}\}_{\lambda \in (0,\lambda^{*})}$ as $\lambda \uparrow \lambda^{*}$. This limiting function $u^{*}$ is referred to as the extremal solution. Since $u^{*}$ is known a priori only to belong to $L^{1}(\Omega)$, the question of its higher regularity plays a central role in the analysis of Gelfand and related problems. 

Brezis and Vázquez \cite{BrezisVazquez1997} posed the question of whether $u^{*}$ belongs to $H_{0}^{1}(\Omega)$ in all dimensions. Subsequently, Brezis \cite{Brezis2003} raised the question of whether $u^{*}$ lies in $L^{\infty}(\Omega)$---and is thus a classical solution---for dimensions $1\le n\le 9$. The sharpness of this dimensional threshold is demonstrated by the example of Joseph and Lundgren \cite{JosephLundgren1973}. For $n \ge 10$ and $f(t) = e^{t}$, they showed that the extremal solution of \eqref{eq:gelfand} in the unit open ball $B_1$ is given by $u^{*}(x) = -2 \ln|x|$ (with $\lambda^* = 2(n-2)$), a function belonging to $H_{0}^{1} \setminus L^{\infty}$. Since the extremal solution arises as the pointwise limit of an increasing family of classical stable solutions, its regularity relies on uniform a priori estimates for such solutions. This observation motivates two general conjectures: stable classical solutions of \eqref{eq:gelfand} with $\lambda = 1$ are uniformly bounded in $H^{1}$ in all dimensions, and are uniformly bounded in $L^{\infty}$ for $1 \le n \le 9$.

Considerable progress was made toward both conjectures before their eventual resolution. Regarding the first conjecture, Nedev \cite{Nedev2000} verified it for dimensions $1\le n\le 5$, and subsequently for all dimensions assuming the convexity of $\Omega$ \cite{Nedev2001}. The case $n=6$ was later settled by Villegas \cite{Villegas2013}. The second conjecture was first confirmed by Nedev \cite{Nedev2000} for dimensions $1 \le n \le 3$. Cabré and Capella \cite{CabreCapella2006} subsequently proved it in the radial framework for all $1 \le n \le 9$. The full resolution of the case $n=4$ is due to Villegas \cite{Villegas2013}; see also \cite{Cabre2010, Cabre2019} for complementary results on convex domains. In higher dimensions, Cabré and Ros-Oton \cite{CabreRosOton2013} established the conjecture for convex domains of double revolution when $5 \le n \le 7$. Moreover, for $n=5$, Cabré, Sanchón, and Spruck \cite{CabreSanchonSpruck2016} proved the conjecture under appropriate hypotheses on the nonlinearity $f$.

Cabr\'e, Figalli, Ros-Oton, and Serra \cite[Thm.\,1.5]{CabreFigalliRosOtonSerra2020} eventually provided a complete affirmative answer to these conjectures, thereby resolving the aforementioned questions posed by Brezis and Vázquez. Specifically, they proved that if $n \ge 1$, $\Omega \subseteq \RR^n$ is a bounded $C^3$ open set, and $f: \RR \to \RR$ is nonnegative, nondecreasing, and convex, then any stable solution $u\in C^2(\Omega) \cap C(\ol{\Omega})$ of \eqref{eq:pde} with $u=0$ on $\partial\Omega$ satisfies
\begin{equation}\label{eq:cfrs-global-gradient}
 \|Du\|_{L^{2+\gamma_n}(\Omega)}\le C(\Omega)\|u\|_{L^1(\Omega)},
\end{equation}
and for $1\le n\le 9$,
\begin{equation}\label{eq:cfrs-global-holder}
 \|u\|_{C^{0,\alpha_n}(\ol\Omega)}\le C(\Omega)\|u\|_{L^1(\Omega)}.
\end{equation}
Notably, the exponents $\gamma_n>0$ and $\alpha_n \in (0,1)$ depend solely on the dimension $n$, and the estimates are independent of the nonlinearity $f$. Since the right-hand side involves only the $L^{1}$ norm of $u$, the extremal solution inherits both estimates (see \cite[Cor.\,1.8]{CabreFigalliRosOtonSerra2020}). Furthermore, \cite[Thm.\,1.2]{CabreFigalliRosOtonSerra2020} establishes interior estimates assuming only the nonnegativity of $f$. More precisely, let $n\ge 1$ and let $B_{1}$ denote the unit open ball in $\RR^n$. If $u \in C^{2}(B_{1})$ is a stable solution of \eqref{eq:pde} for some nonnegative $f \in C_{\loc}^{0,1}(\RR)$, then
\begin{equation}\label{eq:cfrs-gradient}
 \|Du\|_{L^{2+\gamma_n}(B_{1/2})}\le C(n)\|u\|_{L^1(B_1)},
\end{equation}
and, for $1\le n\le9$,
\begin{equation}\label{eq:cfrs-holder}
 \|u\|_{C^{0,\alpha_n}(\ol{B_{1/2}})}\le C(n)\|u\|_{L^1(B_1)}.
\end{equation}

Several developments have followed the work of \cite{CabreFigalliRosOtonSerra2020}. Cabr\'{e} \cite{Cabre2026} established a new proof of \eqref{eq:cfrs-global-gradient}--\eqref{eq:cfrs-holder} by employing a new radial Poincar\'e inequality from \cite[Thm.\,1.4]{Cabre2022}. In contrast to the original proof in \cite{CabreFigalliRosOtonSerra2020}, which relied on delicate compactness arguments, this new approach is quantitative and allows for the explicit computation of the H\"{o}lder exponent $\alpha(n)$ in \eqref{eq:cfrs-global-holder} and \eqref{eq:cfrs-holder}. Furthermore, Erneta \cite{Erneta2023, Erneta2024, Erneta2025} extended several results from \cite{CabreFigalliRosOtonSerra2020} to stable solutions of equations with variable coefficients, and also established boundary estimates on $C^{1,1}$ domains. For results concerning stable solutions of equations involving the $p$-Laplacian, the fractional Laplacian, or MEMS-type nonlinearities, we refer the interested reader to \cite{Miraglio2020,CabreMiraglioSanchon2022,PengShiZou2026}, \cite{CabreSanzPerela2022,SanzPerela2025}, and \cite{BrueraCabre2025,PengVillegas2026}, respectively, as well as the references therein.

Since the proofs of \eqref{eq:cfrs-holder} in \cite{CabreFigalliRosOtonSerra2020,Cabre2026} rely crucially on the nonnegativity of $f$, Cabr\'e \cite[p.\,6]{Cabre2023} and \cite[Sec.\,1.6.1]{Cabre2026} raised the question of whether interior $L^\infty$ estimates can be derived without the assumption $f\ge 0$. So far, only partial results have been available. In the radial setting, Cabr\'e and Capella \cite{CabreCapella2006} obtained such estimates. For general stable solutions, Cabr\'e \cite{Cabre2010,Cabre2019} established such estimates in dimensions $n\le 4$, and Peng, Zhang, and Zhou \cite[Thm.\,1.1]{PengZhangZhou2026} later extended the result to $n\le 5$. In fact, they established the following interior H\"older estimate, in which the right-hand side of \eqref{eq:cfrs-holder} is replaced by $\|Du\|_{L^2}$:
\begin{equation}\label{eq:pzz-holder}
 [u]_{C^{0,\alpha_n}(\ol{B_{1/2}})}\le C(n)\|Du\|_{L^2(B_1)}.
\end{equation}
Regarding conditional results for $1\le n\le 9$, Cabré \cite[Sec.\,2.6]{Cabre2026} established an estimate assuming $f\ge -K$:
\begin{equation}\label{eq:cabre-lower-holder}
 \|u\|_{C^{0,\alpha_n}(\ol{B_{1/2}})}
 \le C(n)\bigl(\|u\|_{L^1(B_1)}+K\bigr).
\end{equation}
More recently, under the condition $f(t) \ge A \min\{0, t\} - K$ for constants $A, K \ge 0$, Peng \cite[Thm.\,1.1]{Peng2024} obtained
\begin{align}\label{zyx5}
		\|u\|_{C^{\alpha_{n,A,K}}(\ol{B_{1/2}})} \le C(n,A,K) \left[ \|u\|_{H^1(B_1)} + 1 \right].
    \end{align}
To the best of our knowledge, establishing interior $L^{\infty}$ estimates in the nonradial case for general $f \in C_{\mathrm{loc}}^{0,1}(\mathbb{R})$ remains an open problem for $6\le n\le 9$.

\textbf{In this article}, we establish sharp interior H\"older estimates for stable classical solutions of \eqref{eq:pde} with general, possibly sign-changing nonlinearities in every dimension $n\le 9$, thus completely resolving Cabr\'e's question. For $n\le 5$, we also give a short alternative proof of \eqref{eq:pzz-holder}, yielding the explicit H\"older exponents
\begin{equation}\label{eee}
 \alpha_2=\alpha_3=\frac12,\qquad
 \alpha_4=\sqrt2-1,\qquad
 \alpha_5=\sqrt3-\frac32.
\end{equation}
For $6\le n\le 9$, the H\"older exponent can also be explicitly determined from our argument, although its value may be small. A useful feature of our proofs is that the main regularity arguments are direct and quantitative.

Beyond the H\"older estimates, we establish higher-integrability estimates for the gradient and $L^1$ bounds for the Hessian in every dimension $n\ge 2$. In higher dimensions, we obtain BMO estimates for $n=10$ and Morrey estimates for $n\ge 11$, extending the corresponding results of \cite{CabreFigalliRosOtonSerra2020,PengZhangZhou2024} to general, possibly sign-changing nonlinearities.

When $f$ is convex, an interior approximation theorem allows us to extend the regularity estimates from stable classical solutions to stable weak solutions. As applications of our regularity theory, we establish several Liouville-type theorems for entire stable solutions across all dimensions, improving several results in \cite{DupaigneFarina2022,PengZhangZhou2024}.

Finally, in Appendix \ref{app:radial-poincare}, we extend Cabr\'e's radial Poincar\'e inequality in \cite{Cabre2022} to general nonlinearities $f$. Even when $f\ge 0$, our argument provides a short alternative proof of this inequality. We also discuss an alternative approach to interior H\"older regularity based on this inequality.

\subsection{Main results}

Throughout this paper, we assume $n\ge 2$, since similar estimates for $n=1$ can be readily established even without the stability assumption. Before stating our main results, we introduce some notation. For $r>0$ and $x\in\RR^n$, let $B_r(x)$ denote the open ball of radius $r$ centered at $x$, and write $B_r:=B_r(0)$. Unless otherwise stated, we assume only that $f\in C_{\mathrm{loc}}^{0,1}(\RR)$ and denote its negative part by $f^-:=\max\{-f,0\}$. Set
\begin{equation}\label{eq:data}
 \mathfrak D(u)=\|u\|_{L^1(B_1\setminus B_{3/4})}+\|f^-(u)\|_{L^1(B_1\setminus B_{3/4})}.
\end{equation}
We use $C(\cdot)$ to denote a generic positive constant, with its dependencies indicated in parentheses. We fix the H\"older exponents
\begin{equation}\label{eq:alpha}
 \alpha_2=\alpha_3=\frac12,\qquad
 \alpha_4=\sqrt2-1,\qquad
 \alpha_5=\sqrt3-\frac32,
 \qquad \alpha_n\in(0,1)\quad(6\le n\le9),
\end{equation}
where, for $6\le n\le9$, $\alpha_n$ is determined by the proof of Proposition \ref{prop:decay} (see \eqref{add11}). Define
\begin{equation}\label{eq:beta}
 \beta_n=\begin{cases}
 2-2\alpha_n,&2\le n\le9,\\
 n-2-2\sqrt{n-1},&n\ge10.
 \end{cases}
\end{equation}
For $1\le q\le\infty$, we set
\begin{equation}\label{eq:gamma}
\gamma_{n,q}=\begin{cases}
 2,&n=2,\ 1\le q\le4/3,\\[2pt]
 (4q-4)/(2-q),&n=2,\ 4/3<q<2,\\[2pt]
 \text{any fixed finite number in }[2,\infty),&n=2,\ 2\le q\le\infty,\\[2pt]
 3\min\{q,2\}-2,&n=3,\\[2pt]
 \displaystyle\max\left\{\frac2{n-1},\,
 \min\left\{\frac4{\beta_n},\left(1+\frac2{\beta_n}\right)q-2\right\}\right\},&n\ge4.
\end{cases}
\end{equation}

\vskip0.1in

Our first theorem provides interior estimates for the gradient and Hessian. Compared with the estimate \eqref{eq:cfrs-gradient} from \cite{CabreFigalliRosOtonSerra2020}, our gradient estimates not only allow general, possibly sign-changing nonlinearities $f$, but also provide explicit higher-integrability exponents.

\begin{theorem}\label{thm:gradient}
Let $n\ge 2$, and let $u\in C^2(B_1)$ be a stable solution of \eqref{eq:pde} in $B_1$. Then, for every $1\le q\le\infty$, with $\gamma_{n,q}$ as defined in \eqref{eq:gamma}, we have
\begin{equation}\label{eq:main-gradient}
 \|Du\|_{L^{2+\gamma_{n,q}}(B_{1/2})}
 +\|D^2u\|_{L^1(B_{1/2})}
 \le C(n,q,\gamma_{n,q})
 \left(\mathfrak D(u)+\|f^-(u)\|_{L^q(B_{5/8})}\right).
\end{equation}
In particular, taking $q=1$ yields
\begin{equation}\label{eq:main-endpoint}
 \|Du\|_{L^{2n/(n-1)}(B_{1/2})}
 +\|D^2u\|_{L^1(B_{1/2})}
 \le C(n)\left(\mathfrak D(u)+\|f^-(u)\|_{L^1(B_{5/8})}\right).
\end{equation}
\end{theorem}

\vskip0.1in

Our second theorem establishes interior H\"older estimates in the full dimensional range $2\le n\le 9$. The first estimate \eqref{eq:main-holder0} extends the result \eqref{eq:pzz-holder} of Peng, Zhang, and Zhou \cite{PengZhangZhou2026} to dimensions $6\le n\le 9$. The second estimate, which is new in every dimension $2\le n\le 9$, controls the full interior H\"older norm using only the $L^1$ norms of $u$ and $f^-(u)$ over the outer annulus $B_1\setminus B_{3/4}$.

\begin{theorem}\label{thm:holder}
Let $2\le n\le 9$, and let $u\in C^2(B_1)$ be a stable solution of \eqref{eq:pde} in $B_1$. Then, with $\alpha_n$ as in \eqref{eq:alpha}, we have
\begin{equation}\label{eq:main-holder0}
 [u]_{C^{0,\alpha_n}(\ol{B_{1/2}})}
 \le C(n)\|Du\|_{L^2(B_{3/4})},
\end{equation}
and
\begin{equation}\label{eq:main-holder}
 \|u\|_{C^{0,\alpha_n}(\ol{B_{1/2}})}
 \le C(n)\mathfrak D(u).
\end{equation}
\end{theorem}

\vskip0.1in

Our third result provides three constructions showing that certain aspects of the four estimates \eqref{eq:main-gradient}--\eqref{eq:main-holder} are sharp.

\begin{theorem}\label{thm:obstructions}
For every $n\ge 2$, the following assertions hold.
\begin{enumerate}
\item[(i)] There exist sequences $f_j\in C^\infty(\RR)$ and stable solutions $u_j\in C^\infty(\RR^n)$ of $-\Delta u_j=f_j(u_j)$ in $\RR^n$ such that
\[
 \sup_j\|u_j\|_{L^\infty(\RR^n)}\le 1,
\]
but
\[
 \|Du_j\|_{L^p(B_{1/2})}\longrightarrow\infty
 \quad(\forall\;1<p\le\infty),
 \qquad
 \|D^2u_j\|_{L^1(B_{1/2})}\longrightarrow\infty,
\]
and $[u_j]_{C^{0,\alpha}(\ol{B_{1/2}})}\to\infty$ for every $0<\alpha\le 1$.

\item[(ii)] There exist sequences $f_j\in C^\infty(\RR)$ and radial stable solutions $u_j\in C^\infty(\ol{B_1})$ of $-\Delta u_j=f_j(u_j)$ in $B_1$ such that $0\le u_j\le 1$ in $B_1$, $u_j=0$ on $\partial B_1$, and
\[
 \sup_j\|Du_j\|_{L^2(B_1)}<\infty.
\]
Moreover, the functions $u_j$ agree on $B_1\setminus B_{3/4}$, $f_j^-(u_j)=0$ on $B_1\setminus B_{3/4}$, and
\[
 \|Du_j\|_{L^p(B_{1/2})}\longrightarrow\infty
 \quad(\forall\;2<p\le\infty),
 \qquad
 \|D^2u_j\|_{L^1(B_{1/2})}\longrightarrow\infty.
\]

\item[(iii)] There exist sequences of nonnegative functions $f_j\in C^\infty(\RR)$ and stable solutions $u_j\in C^\infty(\RR^n)$ of $-\Delta u_j=f_j(u_j)$ in $\RR^n$ such that
\[
 \sup_j\|u_j\|_{W^{1,\infty}(B_1)}<\infty,
 \qquad
 \|D^2u_j\|_{L^p(B_{1/2})}\longrightarrow\infty
 \quad(\forall\;1<p\le\infty).
\]
\end{enumerate}
\end{theorem}

\begin{remark}
Part~(i) shows that the terms involving $f^-(u)$ in \eqref{eq:main-gradient}, \eqref{eq:main-endpoint}, and \eqref{eq:main-holder} are essential. Related examples illustrating the failure of certain $L^p$-$L^q$ estimates for stable solutions were recently constructed by Villegas \cite{Villegas2026}, using more involved arguments. Our construction also provides a negative answer to the question raised by Peng, Zhang, and Zhou \cite[Rem.\,1.4(ii)]{PengZhangZhou2026}.

Part~(ii) shows that the term $\|f^-(u)\|_{L^q(B_{5/8})}$ in \eqref{eq:main-gradient} cannot be omitted. In particular, the annular norm $\|u\|_{L^1(B_1\setminus B_{3/4})}$ alone cannot control $\|Du\|_{L^p(B_{1/2})}$ for any $p>2$, or $\|D^2u\|_{L^1(B_{1/2})}$. By contrast, when $p=2$, the right-hand side can involve only integrals over the outer annulus, as shown in~\eqref{eq:annulus-data}.

Part~(iii) shows that even for uniformly Lipschitz stable solutions with nonnegative nonlinearities, the $L^p$ norm of the Hessian cannot be uniformly bounded for any $p>1$. Thus, the $L^1$ estimate for the Hessian in \eqref{eq:main-gradient} cannot, in general, be improved to an $L^p$ estimate for any $p>1$.
\end{remark}

\vskip0.1in

Our next main result establishes $\BMO$ and Morrey estimates for stable solutions in dimensions $n\ge 10$. Since uniform $L^\infty$ bounds are no longer available for $n\ge 10$, these estimates serve as the natural high-dimensional counterparts to the H\"older regularity obtained for $n\le 9$.

Before stating the theorem, we briefly recall the definitions of the relevant function spaces. Throughout the paper, given an open set $\Omega \subseteq \mathbb{R}^n$, $p \in (0,+\infty)$, and $\beta \in [0,n]$, the Morrey space $M^{p,\beta}(\Omega)$ consists of all measurable functions $u$ on $\Omega$ such that
\begin{equation}\label{eq:morrey-definition}
 \|u\|_{M^{p,\beta}(\Omega)}
 :=\sup_{x\in \Omega,\,r>0}
 \left(r^{\beta-n}\int_{\Omega\cap B_r(x)}|u|^p\right)^{1/p}<+\infty.
\end{equation}
Similarly, the space $\BMO(\Omega)$ of functions of bounded mean oscillation consists of all functions $u \in L^{1}(\Omega)$ satisfying
\[
[u]_{\BMO(\Omega)} := \sup_{x \in \Omega,\, r > 0} \inf_{c \in \RR} \fint_{\Omega \cap B_r(x)} |u-c| < +\infty,
\]
where $\fint_{\Omega \cap B_r(x)}$ denotes the integral average over $\Omega \cap B_r(x)$. Since $[\cdot]_{\BMO(\Omega)}$ defines only a seminorm, we equip $\BMO(\Omega)$ with the norm
\[
 \|u\|_{\BMO(\Omega)} := \|u\|_{L^{1}(\Omega)} + [u]_{\BMO(\Omega)}.
\]

\begin{theorem}\label{thm:morrey}
Let $u\in C^2(B_1)$ be a stable solution of \eqref{eq:pde} in $B_1$.
\begin{enumerate}
\item[(i)] If $n\ge2$ and $\beta_n\le\beta\le n$, then
\[
 \|Du\|_{M^{2,\beta}(B_{1/2})}\le C(n,\beta)\mathfrak D(u).
\]

\item[(ii)] If $2\le n\le10$, then
\[
 \|u\|_{\BMO(B_{1/2})}\le C(n)\mathfrak D(u).
\]

\item[(iii)] If $n=10$ and $2n/(n-2)\le p<\infty$, or if $n\ge11$ and $2n/(n-2)\le p\le p_n:=2\beta_n/(\beta_n-2)$, then
\[
 \|u\|_{M^{p,\,2p/(p-2)}(B_{1/2})}\le C(n,p)\mathfrak D(u).
\]
\end{enumerate}
\end{theorem}

\begin{remark}
Several remarks concerning Theorem~\ref{thm:morrey} are in order:
\begin{enumerate}
  \item[(a)] For $n \ge 11$, the ranges of the exponents $\beta$ in~(i) and $p$ in~(iii) are sharp; we refer to \cite[Sec.\,6]{BrezisVazquez1997} and \cite[Sec.\,1.3]{CabreFigalliRosOtonSerra2020} for a detailed discussion.
  
  \item[(b)] Theorem~\ref{thm:morrey} extends \cite[Thms.\,1.9 and 7.1]{CabreFigalliRosOtonSerra2020} and \cite[Thm.\,1.2]{PengZhangZhou2024} in two directions: it accommodates general nonlinearities $f$, and the right-hand side $\mathfrak D(u)$ requires control only on an annulus rather than the full ball.

  \item[(c)] In dimension $n=10$, combining the $\BMO$ bound in~(ii) with the John--Nirenberg inequality \cite{JohnNirenberg1961} yields the optimal $L^p$ estimate
  \[
  \|u\|_{L^p(B_{1/2})}
  \le C(n,p)\mathfrak D(u),\qquad\forall\, 0<p<\infty
  \]
  (recall from the introduction that the endpoint $p=\infty$ fails for $n=10$). In a forthcoming work \cite{HuangZhouZouLp}, we establish the corresponding optimal $L^p$ estimate in higher dimensions $n\ge 11$:
  \[
  \|u\|_{L^p(B_{1/2})}
  \le C(n,p)\mathfrak D(u),\qquad\forall\, 0<p<\frac{2n}{\beta_n-2},
  \]
  where the optimality of the range of $p$ was shown by Cabr\'e and Capella \cite[Rem.\,1.7]{CabreCapella2006}. Moreover, we demonstrate in \cite{HuangZhouZouLp} that the term $f^-(u)$ in $\mathfrak D(u)$ cannot be dropped in general. This completely resolves a question posed in \cite{CabreFigalliRosOtonSerra2020}, which had remained open even for nonnegative nonlinearities $f\ge 0$.
\end{enumerate}
\end{remark}
  
\vskip0.1in

The next two theorems concern interior approximation and regularity for stable weak solutions, respectively. Theorem \ref{thm:approximation} improves upon \cite[Thm.\,1.3]{PengZhangZhou2026} in two respects: it does not require $f$ to be nondecreasing, and the radius $r_0$ is independent of $f$. This refinement allows us to establish interior regularity estimates for stable weak solutions under the sole additional assumption that $f$ is convex.

\begin{theorem}\label{thm:approximation}
Let $u\in H_{\loc}^{1}(B_1)$ be a stable weak solution of \eqref{eq:pde} in $B_1$, and assume that $f$ is convex. Then there exist $r_0=r_0(n)\in(0,1/4)$, a sequence of convex functions $f_k\in C^{0,1}_{\loc}(\RR)$, and a sequence of functions
\[
 u_k\in H^1(B_{r_0})\cap C^2(B_{r_0}),\qquad k\ge 1,
\]
satisfying the following properties:
\begin{enumerate}
\item[(i)] $f_k=f$ on $(-\infty,k]$, $f_k\le f_{k+1}\le f$ in $\RR$, and $f_k\uparrow f$ pointwise in $\RR$;

\item[(ii)] $u_k-u\in H_0^1(B_{r_0})$, and $u_k$ is a classical stable solution of
\[
 -\Delta u_k=f_k(u_k)\qquad\hbox{in }B_{r_0};
\]

\item[(iii)] $-M\le u_k\le u_{k+1}\le u$ a.e. in $B_{r_0}$ for some constant $M$, $u_k\uparrow u$ a.e. in $B_{r_0}$, $f_k(u_k)\to f(u)$ in $L^1(B_{r_0})$, and
\[
 u_k\longrightarrow u\qquad\hbox{strongly in }H^1(B_{r_0}).
\]
\end{enumerate}
\end{theorem}

\begin{theorem}\label{thm:weak}
Let $u\in H^1_{\loc}(B_1)$ be a stable weak solution of \eqref{eq:pde} in $B_1$, and assume that $f$ is convex. Then the conclusions of Theorems~\ref{thm:gradient}, \ref{thm:holder}, and~\ref{thm:morrey} remain valid, with the same exponents, when $\mathfrak D(u)$ on the right-hand side of each estimate is replaced by
\[
 \|u\|_{L^1(B_1)}+\|f^-(u)\|_{L^1(B_1)}.
\]
\end{theorem}

A natural next step is to investigate the boundary regularity of stable solutions. We address this problem in a forthcoming work \cite{HuangZhouZouBRSSVC} in the setting of semilinear elliptic equations with variable coefficients and sign-changing nonlinearities.

\vskip0.1in

Our final result establishes Liouville-type theorems for entire stable solutions in $\RR^n$. Rather than imposing pointwise growth or decay conditions at infinity, our approach relies only on $L^1$ bounds for $u$, $u^-$, and $f^-(u)$ over dyadic annuli. Liouville-type results play a central role in the theory of stable solutions and have been extensively studied. Without attempting an exhaustive survey, we refer the reader to \cite{CabreCapella2004,Farina200705,Farina200707,Villegas2007,DupaigneFarina2010,FarinaValdinoci2010,Dupaigne2011,FarinaSciunzi2017,DupaigneFarina2022,LiuWangWeiWu2024,PengZhangZhou2024} and the references therein.

\begin{theorem}\label{thm:liouville}
Let $u\in C^2(\RR^n)$ be a stable solution of \eqref{eq:pde} in $\RR^n$, and set $A_R:=B_{2R}\setminus\ol{B_R}$ for $R>0$. Then $u$ is constant if any of the following conditions holds as $R\to\infty$:
\begin{enumerate}
\item[(i)] $n=2$, and
\[
 \|u^-\|_{L^1(A_R)}=o(R^3),
 \qquad
 \|f^-(u)\|_{L^1(A_R)}=o(R).
\]

\item[(ii)] $3\le n\le9$, and
\[
 \|u^-\|_{L^1(A_R)}=o(R^{n+\alpha_n}),
 \qquad
 \|f^-(u)\|_{L^1(A_R)}=o(R^{n-2+\alpha_n}).
\]

\item[(iii)] $n=10$, and
\[
 \|u^-\|_{L^1(A_R)}=o(R^n\log R),
 \qquad
 \|f^-(u)\|_{L^1(A_R)}=o(R^{n-2}).
\]

\item[(iv)] $n\ge11$, and, with $\kappa_n:=(2-\beta_n)/2<0$,
\[
 \|u\|_{L^1(A_R)}=o(R^{n+\kappa_n}),
 \qquad
 \|f^-(u)\|_{L^1(A_R)}=o(R^{n-2+\kappa_n}).
\]
\end{enumerate}
\end{theorem}

\begin{remark}
We make several remarks on Theorem~\ref{thm:liouville}.
\begin{enumerate}
\item[(a)] For $2\le n\le9$, Theorem~\ref{thm:liouville} extends \cite[Thm.\,2]{DupaigneFarina2022} by allowing general, possibly sign-changing nonlinearities $f$ and explicit polynomial growth bounds on $u^-$. In particular, when $n=2$, the growth condition $\|u^-\|_{L^1(A_R)}=o(R^3)$ is sharp, as shown by the harmonic function $u(x)=x_1$ with $f\equiv0$.

\item[(b)] For $n=10$, Theorem~\ref{thm:liouville} generalizes both \cite[Thm.\,1]{DupaigneFarina2022}, which assumes $f\ge0$ and boundedness of $u^-$, and \cite[Thm.\,1.5]{PengZhangZhou2024}, which assumes $f\ge0$ and imposes the two-sided condition $\|u\|_{L^1(A_R)}=o(R^n\log R)$. Moreover, the growth condition $\|u^-\|_{L^1(A_R)}=o(R^n\log R)$ is sharp; see \cite[Exam.\,3.1]{Villegas2007}.

\item[(c)] For $n\ge11$, Theorem~\ref{thm:liouville} extends \cite[Thm.\,1.5]{PengZhangZhou2024} to general, possibly sign-changing nonlinearities $f$. As shown in \cite[Exam.\,3.1]{Villegas2007}, the growth condition $\|u\|_{L^1(A_R)}=o(R^{n+\kappa_n})$ is also sharp.

\item[(d)] Finally, we emphasize that Liouville-type theorems for stable solutions in $\RR^n$ generally require suitable lower bounds on both $f$ and $u$, and that the relevant conditions depend strongly on the dimension. We refer the reader to \cite[Exam.\,3.1]{Villegas2007}, \cite[Rem.\,3]{DupaigneFarina2022}, and \cite[p.\,3339]{PengZhangZhou2024} for further discussion and counterexamples.
\end{enumerate}
\end{remark}

\subsection{Strategy of the proofs}

We now outline the main ideas of our proofs. We begin by establishing the fundamental gradient and Hessian estimates in Lemma~\ref{lem:endpoint}. Part of our argument builds on the level-set method developed in \cite{Cabre2010,CabreFigalliRosOtonSerra2020}, which combines the geometric inequality for stable solutions with the Michael--Simon Sobolev inequality. In our setting, allowing the nonlinearity to change sign requires new arguments to handle its negative part. Together with Proposition~\ref{prop:annulus}, these estimates allow us to control the interior gradient energy in terms of the $L^1$ norms of $u$ and $f^-(u)$ over outer annuli.

A key ingredient in our regularity theory is the energy decay estimate in Proposition~\ref{prop:decay}, whose right-hand side involves only the gradient energy. In dimensions $2\le n\le5$, we develop a direct approach based on rotational vector fields. Specifically, for $3\le n\le5$, we consider the rotational derivatives $W_{ij}=x_i\partial_j u-x_j\partial_i u$. Combining the equations satisfied by these functions with the stability inequality yields a decay estimate for the tangential gradient. Averaging over the centers of rotation then gives the corresponding decay estimate for the full gradient. The two-dimensional case follows by a dimension-lifting argument. This approach also yields the explicit H\"older exponents in \eqref{eq:alpha}.

The main analytical difficulty arises in dimensions $6\le n\le9$, where we establish the quantitative rigidity result in Proposition~\ref{prop:rigidity}. Although its statement parallels \cite[Lem.\,3.1]{CabreFigalliRosOtonSerra2020}, our proof takes a substantially different approach, avoiding both compactness arguments and the nonnegativity assumption on $f$. We first combine the stability inequality with the Riesz representation theorem in Lemma~\ref{lem:reaction} to express certain terms involving $f(u)$ as divergences of vector fields with small $L^1$ norms. Next, Lemma~\ref{lem:primitive} combines the weak $(1,1)$ estimate for singular integrals with a $W^{1,1}$ bound to show that a suitable auxiliary function is close to a constant in $L^1$. Applying these results to a primitive $F$ of $f$, we obtain control of the gradient energy on sets of small measure.

To estimate the remaining energy, we consider the vector field $V=|Du|Du$. Using the divergence representation from Lemma~\ref{lem:reaction}, we derive estimates for the translation differences of $V$. We then mollify $V$ and construct a suitable regularized inverse of the map $p\mapsto|p|p$ to approximate $Du$. This reduces the remaining energy estimate to controlling the divergence of the approximating vector field. A careful analysis of this divergence, together with the preceding estimates, completes the proof of Proposition~\ref{prop:rigidity}.

For $6\le n\le9$, we deduce Proposition~\ref{prop:decay} by combining the radial energy estimates with averaging over the scaling parameter to find a scale at which the hypotheses of Proposition~\ref{prop:rigidity} hold. Proposition~\ref{prop:rigidity} then yields a fixed-scale contraction of the gradient energy, which gives the desired decay estimate by iteration. Related scaling and averaging ideas also appear, in a different setting, in Cabr\'e's treatment of boundary radial Poincar\'e inequalities~\cite[Thm.\,1.9]{Cabre2022}. For $n\ge10$, the corresponding energy bounds follow from the radial estimates in Lemma~\ref{lem:radial}.

With these energy estimates in hand, Theorem~\ref{thm:holder} follows from a Morrey-type embedding theorem. To prove Theorem~\ref{thm:gradient}, we combine the energy decay with the Morrey--BV interpolation inequality in Lemma~\ref{lem:morrey-bv}, a weighted Hessian estimate, and Simon's absorbing lemma. Theorem~\ref{thm:morrey} then follows from these estimates together with Poincar\'e inequalities, the John--Nirenberg inequality, and Riesz-potential estimates in Morrey spaces. The proof of this last theorem is partly inspired by \cite{CabreFigalliRosOtonSerra2020,Peng2024,PengZhangZhou2024}.

To demonstrate the sharpness of our regularity estimates, we construct several explicit families of stable solutions in Theorem~\ref{thm:obstructions}. For Theorem~\ref{thm:approximation}, we adapt the local approximation scheme of Peng, Zhang, and Zhou \cite[Thm.\,1.3]{PengZhangZhou2026}, which assumes that $f$ is nondecreasing and convex. By introducing a suitable affine lower bound for $f$ and refining the comparison argument, we remove the assumption that $f$ is nondecreasing. Moreover, the approximation radius can be chosen to depend only on the dimension. Theorem~\ref{thm:weak} then follows by applying our regularity estimates for classical solutions to the approximating sequence and passing to the limit.

Finally, to prove Theorem~\ref{thm:liouville}, we first establish several growth criteria that convert $L^1$ bounds on $u^-$ and $f^-(u)$ over annuli into estimates for $u$ and its gradient over large balls. In dimension $n=2$, we combine the resulting gradient energy bound with a one-dimensional symmetry argument, while for $3\le n\le9$ the conclusion follows from Proposition~\ref{prop:decay}. For $n\ge10$, we use high-dimensional energy estimates from Peng, Zhang, and Zhou \cite{PengZhangZhou2024}. In the critical dimension $n=10$, we further refine their annular argument to accommodate $f^-(u)$ and to replace their two-sided growth condition on $u$ by a one-sided condition on $u^-$. Together, these arguments yield the Liouville-type theorems in all dimensions.

\subsection{Organization of the paper}

The remainder of the paper is organized as follows. In Section~\ref{sec2}, we establish the auxiliary estimates and inequalities needed in the subsequent proofs, including the key energy decay and rigidity results. Section~\ref{sec3} contains the proofs of Theorems~\ref{thm:gradient}, \ref{thm:holder}, and~\ref{thm:obstructions}. In Section~\ref{sec4}, we prove Theorem~\ref{thm:morrey} on BMO and Morrey estimates. Section~\ref{sec5} is devoted to the approximation theorem and the regularity of stable weak solutions, namely Theorems~\ref{thm:approximation} and~\ref{thm:weak}. In Section~\ref{sec6}, we establish the Liouville-type theorems in Theorem~\ref{thm:liouville}. Finally, Appendix~\ref{app:radial-poincare} extends Cabr\'e's radial Poincar\'e inequality to general nonlinearities.

\section{Auxiliary results}\label{sec2}
In this section, we collect auxiliary results that will be used in the proofs of Theorems~\ref{thm:gradient}--\ref{thm:obstructions} in Section \ref{sec3}. In the first subsection, we provide basic estimates and functional inequalities. The second subsection is devoted to energy-decay estimates.

\subsection{Basic estimates and functional inequalities}\label{sec2.1}
Our first lemma lists various useful identities and inequalities for stable solutions of \eqref{eq:pde}.
\begin{lemma}\label{lem:identities}
Let $u\in C^2(\Omega)$ be a stable solution of \eqref{eq:pde} in an open set $\Omega\subseteq\RR^n$. Put $p=Du$, $w=|p|$, and define
\[
 \mathcal A_u^2=|D^2u|^2-\frac{|D^2u\,p|^2}{w^2},\qquad
 B=D^2u+f(u)\frac{p\otimes p}{w^2}
 \quad\hbox{on }\{w>0\},
\]
with both quantities equal to zero on $\{w=0\}$. Then
\begin{equation}\label{eq:geometric}
 \int_\Omega\mathcal A_u^2\eta^2\le\int_\Omega w^2|D\eta|^2
 \qquad(\forall\;\eta\in C_c^{0,1}(\Omega)),
\end{equation}
and, almost everywhere,
\begin{equation}\label{eq:matrix}
 \tr B=0,\quad |B|^2\le n\mathcal A_u^2,\quad
 |D^2u|\le|\Delta u|+\sqrt n\,\mathcal A_u.
\end{equation}
If we set $F'=f$ and $K=w^2/2+F(u)$, then
\begin{equation}\label{eq:stress}
 DK=Bp,\qquad |DK|\le\sqrt{n-1}\,w\mathcal A_u,\qquad
 \Div(p\otimes p-w^2I+KI)=0\quad\hbox{in }\mathcal D'(\Omega).
\end{equation}
Here, for a matrix field $M=(M_{ij})$ on $\Omega$, $\Div M$ is a vector field and we write
$(\Div M)_i:=\sum_{j=1}^n\partial_jM_{ij}$ for its row-wise divergence. For $z=x\cdot Du$, it holds that
\begin{equation}\label{eq:jacobi}
 (-\Delta-f'_-(u))\partial_i u=0,\qquad
 (-\Delta-f'_-(u))z=2f(u),
\end{equation}
\begin{equation}\label{eq:dilation}
 \begin{split}
 Q_u(z\eta)={}&\int_\Omega z^2|D\eta|^2
 +\int_\Omega w^2\bigl((2-n)\eta^2-x\cdot D(\eta^2)\bigr)\\
 &+2\int_\Omega z\,p\cdot D(\eta^2).
 \end{split}
\end{equation}
\end{lemma}

\begin{proof}
The inequality \eqref{eq:geometric} follows from \cite[Lem.\,2.3]{CabreFigalliRosOtonSerra2020}, while the estimates \eqref{eq:jacobi} and \eqref{eq:dilation} follow from \cite[Lem.\,A.3]{CabreFigalliRosOtonSerra2020} and \cite[Lem.\,2.1]{CabreFigalliRosOtonSerra2020}, respectively.

It remains to prove \eqref{eq:matrix} and \eqref{eq:stress}. On $\{w=0\}$, we have $D^2u=0$ and hence $f(u)=0$ almost everywhere. Therefore \eqref{eq:matrix} and \eqref{eq:stress} clearly hold.

At a point where $w>0$, we set $e=p/w$ and decompose $D^2u$ relative to $\RR e\oplus e^\perp$ as
\[
 D^2u=\begin{pmatrix}a&b^T\\ b&H\end{pmatrix}.
\]
Since $\Delta u=-f(u)$, we have $\mathcal A_u^2=|b|^2+|H|^2$ and
\[
 B=\begin{pmatrix}-\tr H&b^T\\ b&H\end{pmatrix}.
\]
Hence $\tr B=0$ and $|B|^2=(\tr H)^2+2|b|^2+|H|^2\le n\mathcal A_u^2$. Moreover, $D^2u=B-f(u)e\otimes e$, and therefore $|D^2u|\le|\Delta u|+\sqrt n\,\mathcal A_u$. The estimate \eqref{eq:matrix} is proved.

Differentiating $K=w^2/2+F(u)$ gives $DK=D^2u\,p+f(u)p=Bp$. Since $|Be|^2=(\tr H)^2+|b|^2\le(n-1)\mathcal A_u^2$, it follows that $|DK|\le\sqrt{n-1}\,w\mathcal A_u$. Finally, using $\Delta u=-f(u)$, we obtain
\[
 \begin{aligned}
 \bigl[\Div(p\otimes p-w^2I+KI)\bigr]_i
 &=\sum_j\partial_j(\partial_i u\,\partial_j u)-\partial_i|Du|^2+\partial_iK\\
 &=\partial_i u\,(\Delta u+f(u))=0.
 \end{aligned}
\]
This proves \eqref{eq:stress}.
\end{proof}

Our next lemma provides a hypersurface Gagliardo--Nirenberg inequality. It will be used in a level set argument in the proof of Lemma \ref{lem:endpoint} below.

\begin{lemma}\label{lem:hypersurface-sobolev}
Let $m\ge1$, and let $\Sigma\subseteq\RR^{m+1}$ be a smooth $m$-dimensional hypersurface. Then for every compactly supported smooth function $\phi$ on $\Sigma$,
\begin{equation}\label{eq:surface-gn}
 \int_\Sigma |\phi|^{2+4/m}
 \le C(m)\left(\int_\Sigma\phi^2\right)^{2/m}
       \int_\Sigma\bigl(|D_\Sigma\phi|^2+\mathcal H^2\phi^2\bigr),
\end{equation}
where $\mathcal H$ denotes the trace mean curvature.
\end{lemma}

\begin{proof}
For $m\ge2$, \eqref{eq:surface-gn} follows easily from the Michael--Simon Sobolev inequality \cite[Thm.\,2.1]{MichaelSimon1973}. Specifically, for $m\geq 3$, one uses the $W^{1,2}$-form and interpolates it with $L^2$. For $m=2$, one applies the $W^{1,1}$-form to $\phi^2$.

Next we suppose $m=1$. We first claim that, on each connected component of $\Sigma$,
\begin{equation}\label{eq:curve-sup}
 \|\phi\|_{L^\infty(\Sigma)}^2
 \le C\|\phi\|_{L^2(\Sigma)}
 \left(\int_\Sigma\bigl(|D_\Sigma\phi|^2+\mathcal H^2\phi^2\bigr)\right)^{1/2}.
\end{equation}
If the component is nonclosed, then $\phi$ vanishes outside a compact subset of it. Hence, for every $x$ on the component, it holds that
\[
 |\phi(x)|^2\le\int_\Sigma|D_\Sigma(\phi^2)|
 \le2\|\phi\|_{L^2(\Sigma)}\|D_\Sigma\phi\|_{L^2(\Sigma)}.
\]
If the component is closed, set $\psi=\phi^2$ and $m_\psi=\min_\Sigma\psi$. Then $\psi(x)\le m_\psi+\int_\Sigma|D_\Sigma\psi|$ for every $x\in\Sigma$. Since the total absolute curvature of a closed curve is at least $2\pi$,
\[
 2\pi m_\psi\le m_\psi\int_\Sigma|\mathcal H|
 \le\int_\Sigma|\mathcal H|\phi^2.
\]
Consequently, we get
\[
 \begin{aligned}
 \|\phi\|_{L^\infty(\Sigma)}^2
 &\le\int_\Sigma|D_\Sigma(\phi^2)|
   +\frac1{2\pi}\int_\Sigma|\mathcal H|\phi^2\\
 &\le C\|\phi\|_{L^2(\Sigma)}
 \left(\int_\Sigma\bigl(|D_\Sigma\phi|^2+\mathcal H^2\phi^2\bigr)\right)^{1/2}.
 \end{aligned}
\]
Thus \eqref{eq:curve-sup} holds in either case. On each component,
$\int_\Sigma|\phi|^6\le \|\phi\|_{L^\infty(\Sigma)}^4\int_\Sigma\phi^2$, which gives \eqref{eq:surface-gn} for $m=1$. Summing over the components proves the general case.
\end{proof}

The following lemma provides a useful interior estimate for $\Delta u$ and $Du$ in terms of the $L^1$ norms of $u$ and $(\Delta u)^+$. Although this estimate may be standard, we have not found a convenient reference in the literature, so we include a proof for completeness.

\begin{lemma}\label{lem:poisson-mass}
Let $n\ge2$, $0<r<R$, and $1\le s<n/(n-1)$. Every $u\in W^{2,1}(B_R)$ satisfies
\begin{equation}\label{eq:poisson-mass}
 \|\Delta u\|_{L^1(B_r)}+\|Du\|_{L^s(B_r)}
 \le C(n,s,r,R)\left(\|u\|_{L^1(B_R)}
       +\|(\Delta u)^+\|_{L^1(B_R)}\right).
\end{equation}
\end{lemma}

\begin{proof}
Choose $r<r_1<r_2<R$ and a cutoff function $\chi\in C_c^\infty(B_R)$ equal to one on $B_{r_2}$. Since $|\Delta u|=-\Delta u+2(\Delta u)^+$, integration by parts gives
\[
 \|\Delta u\|_{L^1(B_{r_2})}
 \le \|\Delta\chi\|_{L^\infty(B_R)}\|u\|_{L^1(B_R)}
      +2\|(\Delta u)^+\|_{L^1(B_R)}.
\]
Take $\zeta\in C_c^\infty(B_{r_2})$ equal to one on $B_{r_1}$, let $\Gamma$ be the fundamental solution of the Laplacian, and set $v=\Gamma*(\zeta\Delta u)$. Then $h=u-v$ is harmonic in $B_{r_1}$. Since $\Gamma\in L^1_{\rm loc}$ and $D\Gamma\in L^s_{\rm loc}$ for $s<n/(n-1)$, Young's inequality yields
\[
 \|v\|_{L^1(B_{r_1})}+\|Dv\|_{L^s(B_{r_1})}
 \le C(n,s,r_1,r_2)\|\Delta u\|_{L^1(B_{r_2})}.
\]
A standard interior estimate for the harmonic function $h$ gives
\[
 \|Dh\|_{L^s(B_r)}\le C(n,s,r,r_1)\|h\|_{L^1(B_{r_1})}
 \le C\bigl(\|u\|_{L^1(B_R)}+\|\Delta u\|_{L^1(B_{r_2})}\bigr).
\]
Combining the above three estimates proves \eqref{eq:poisson-mass}.
\end{proof}

Our next lemma establishes key gradient and Hessian estimates. Its proof combines the preceding three lemmas with a level-set argument.

\begin{lemma}\label{lem:endpoint}
Let $n\ge2$ and $0<r<R<1$. Every classical stable solution $u$ of \eqref{eq:pde} in $B_1$ satisfies, with $p_0=2n/(n-1)$,
\begin{equation}\label{eq:endpoint-estimate}
 \begin{aligned}
 \|Du\|_{L^{p_0}(B_r)}+\|D^2u\|_{L^1(B_r)}
 &\le C(n,r,R)\left(\inf_{c\in\RR}\|u-c\|_{L^1(B_R)}+\|f^-(u)\|_{L^1(B_R)}\right)\\
 &\le C(n,r,R)\left(\|Du\|_{L^1(B_R)}+\|f^-(u)\|_{L^1(B_R)}\right)\\
 &\le C(n,r,R)\left(\|Du\|_{L^2(B_R)}+\|f^-(u)\|_{L^1(B_R)}\right).
 \end{aligned}
\end{equation}
\end{lemma}

\begin{proof}
Choose a cutoff function $\eta\in C_c^\infty(B_{(r+R)/2})$, equal to one on $B_r$, with $0\le\eta\le1$. Put $w=|Du|$. For every $t\in\RR$, on the regular part $\Sigma_t=\{x\in B_R:\;u(x)=t,\ w(x)>0\}$ of the level set, we set
\[
 F_\eta(t)=\int_{\Sigma_t}\eta^2w,\qquad
 B_\eta(t)=\int_{\Sigma_t}
 \left(\frac{\eta^2\mathcal A_u^2}{w}+w|D\eta|^2\right).
\]
The coarea formula and \eqref{eq:geometric} give
\begin{align}\label{add1}
    \int_\RR F_\eta(t)\,dt\le\int_{B_R}w^2,
 \qquad \int_\RR B_\eta(t)\,dt\le C(n,r,R)\int_{B_R}w^2.
\end{align}
For $\psi\in C_c^1(\RR)$, the coarea formula gives $\int_\RR F_\eta(t)\psi(t)\,dt=\int_{B_R}\eta^2w^2\psi(u)$. Hence
\[
 \begin{aligned}
 -\int_\RR F_\eta(t)\psi'(t)\,dt
 &=-\int_{B_R}\eta^2Du\cdot D(\psi(u))
 =\int_{B_R}\operatorname{div}(\eta^2Du)\,\psi(u)\\
 &=\int_{B_R}\bigl(-\eta^2f(u)+2\eta Du\cdot D\eta\bigr)\psi(u).
 \end{aligned}
\]
Thus, in the sense of distributions on $\RR$, we have $DF_\eta=u_\#[\operatorname{div}(\eta^2Du)\,dx]$, where $u_\#$ denotes the push-forward by $u$. Since $\operatorname{div}(\eta^2Du)\le\eta^2f^-(u)+2\eta w|D\eta|$, the positive variation satisfies $(DF_\eta)^+(\RR)\le\int_{B_R}\eta^2f^-(u)+2\int_{B_R}\eta w|D\eta|$. As $F_\eta\ge0$ and has compact support, we have $\|F_\eta\|_{L^\infty(\RR)}\le(DF_\eta)^+(\RR)$. Consequently,
\begin{equation}\label{eq:level-flux}
 \|F_\eta\|_{L^\infty(\RR)}
 \le C(n,r,R)\left(\|Du\|_{L^2(B_R)}+\|f^-(u)\|_{L^1(B_R)}\right).
\end{equation}

As in the proof of \eqref{eq:matrix} and \eqref{eq:stress}, one can compute that
$|D_{\Sigma_t}w|^2=|D^2u\,e|^2-|D^2u[e,e]|^2\le\mathcal A_u^2$ and
$w^2\mathcal H^2=(\Delta u-D^2u[e,e])^2\le(n-1)\mathcal A_u^2$, where $e=Du/w$. Choose a nondecreasing function $\vartheta\in C^\infty([0,\infty))$ such that $0\le\vartheta\le1$, $\vartheta=0$ on $[0,1]$, $\vartheta=1$ on $[2,\infty)$, and $|\vartheta'|\le C$, and put $\phi_\varepsilon=\eta\sqrt w\,\vartheta(w/\varepsilon)$. Direct computation yields
\[
 \begin{aligned}
 |D_{\Sigma_t}\phi_\varepsilon|^2
 &\le Cw|D\eta|^2
 +C\eta^2\left(\frac{\vartheta(w/\varepsilon)}{\sqrt w}
 +\frac{\sqrt w}{\varepsilon}|\vartheta'(w/\varepsilon)|\right)^2
 |D_{\Sigma_t}w|^2\\
 &\le Cw|D\eta|^2+C\frac{\eta^2\mathcal A_u^2}{w}.
 \end{aligned}
\]
As $\phi_\varepsilon$ is compactly supported in $\Sigma_t$ and the curvature term satisfies $\mathcal H^2\phi_\varepsilon^2\le(n-1)\eta^2\mathcal A_u^2/w$, we can apply Lemma~\ref{lem:hypersurface-sobolev} to deduce
\[
 \int_{\Sigma_t}\eta^{2+4/m}w^{1+2/m}\vartheta(w/\varepsilon)^{2+4/m}
 \le C(m)F_\eta(t)^{2/m}B_\eta(t),\qquad m=n-1.
\]
Letting $\varepsilon\downarrow0$, applying Fatou's lemma, and then integrating in $t$ using the coarea formula, we obtain
\[
 \int_{B_r}w^{2+2/m}
 \le C(n)\|F_\eta\|_{L^\infty(\RR)}^{2/m}
       \int_\RR B_\eta(t)\,dt.
\]
From \eqref{add1} and \eqref{eq:level-flux} we further deduce
\[
 \int_{B_r}w^{2+2/(n-1)}
 \le C(n,r,R)(\|Du\|_{L^2(B_R)}+\|f^-(u)\|_{L^1(B_R)})^{2/(n-1)}\|Du\|_{L^2(B_R)}^2.
\]
By Young's inequality we obtain
\begin{align}\label{add2}
    \|Du\|_{L^{p_0}(B_r)}\leq C\|Du\|_{L^2(B_R)}+C\|f^-(u)\|_{L^1(B_R)}.
\end{align}

For the Hessian estimate, take $\chi\in C_c^\infty(B_R)$ nonnegative and equal to one on $B_r$. Thanks to the identity
\begin{equation}\label{eq:laplace-sign}
 |\Delta u|=-\Delta u+2f^-(u),
\end{equation}
we have $\int_{B_R}\chi|\Delta u|=\int_{B_R}Du\cdot D\chi+2\int_{B_R}\chi f^-(u)$. Combining this with \eqref{eq:matrix} and \eqref{eq:geometric} yields
\begin{equation}\label{eq:energy-source}
 \|D^2u\|_{L^1(B_r)}
 \le C(n,r,R)\left(\|Du\|_{L^2(B_R)}+\|f^-(u)\|_{L^1(B_R)}\right).
\end{equation}

It remains to replace the energy on the right of \eqref{add2} and \eqref{eq:energy-source} by $\inf_{c\in\RR}\|u-c\|_{L^1(B_R)}$. Fix $c\in\RR$ and set $D_c=\|u-c\|_{L^1(B_R)}+\|f^-(u)\|_{L^1(B_R)}$. Choose $1<s<n/(n-1)$. By applying Lemma~\ref{lem:poisson-mass} to $u-c$ and interpolating $L^2$ between $L^s$ and $L^{p_0}$, we can deduce from \eqref{add2} that for any $\varepsilon>0$,
\begin{align}\label{add3}
     \|Du\|_{L^{2}(B_r)}\leq \varepsilon\|Du\|_{L^{2}(B_R)}+C(n,r,R,\varepsilon)D_c.
\end{align}
For every ball $B\subseteq B_R$, set $\sigma(B)=\int_B|Du|^2$. Combining \eqref{add3} and a standard scaling argument, we get that, for every ball $B_\rho(y)\subseteq B_R$,
\begin{equation}\label{eq:simon-pattern}
 \rho^{n+2}\sigma(B_{\rho/4}(y))
 \le \varepsilon\rho^{n+2}\sigma(B_\rho(y))
 +C(n,\varepsilon)\left(\|u-c\|_{L^1(B_\rho(y))}
       +\rho^2\|f^-(u)\|_{L^1(B_\rho(y))}\right)^2.
\end{equation}
Since $R<1$, the last term is bounded by $C(n,\varepsilon) D_c^2$. Taking $\varepsilon$ sufficiently small, we can apply Simon's lemma \cite{Simon1997} (see also \cite[Lem.\,A.4]{CabreFigalliRosOtonSerra2020}) to deduce that $\|Du\|_{L^{2}(B_{R/2})}\leq C(n,R)D_c$. Combining this with \eqref{add2}, \eqref{eq:energy-source}, and a further scaling and covering argument concludes the proof.
\end{proof}

The following result shows that the $L^2$ norm of $Du$ over an inner ball can be controlled by integrals over an outer annulus.

\begin{proposition}\label{prop:annulus}
Let $n\ge2$ and $0<r<1$. Every classical stable solution $u$ of \eqref{eq:pde} in $B_1$ satisfies
\begin{equation}\label{eq:annulus-energy}
 \int_{B_r}|Du|^2\le C(n,r)\int_{B_1\setminus B_r}|Du|^2.
\end{equation}
Moreover, for any $0<r,s<1$ and any $c\in\RR$,
\begin{equation}\label{eq:annulus-data}
  \|Du\|_{L^2(B_r)}
  \le C(n,r,s)\left(\|u-c\|_{L^1(B_1\setminus B_s)}
  +\|f^-(u)\|_{L^1(B_1\setminus B_s)}\right).
\end{equation}
\end{proposition}

\begin{proof}
Choose $r<s_1<s_2<1$ and let $\chi\in C_c^\infty(B_{s_1})$ be equal to one on $B_r$. Let $p,w,K$ be as in Lemma \ref{lem:identities}, and set $S=\int_{B_r}w^2$, $T=\int_{B_1\setminus B_r}w^2$. By \eqref{eq:geometric}, we have $\int_{B_{s_1}}\mathcal A_u^2\le C(n,r)T$. Let $\mathbb T=p\otimes p-w^2I+KI$. Then $\Div\mathbb T=0$ by \eqref{eq:stress}. Testing this identity with $x\chi$ gives $0=-\int_{B_{s_1}}\mathbb T:D(x\chi)$. Direct expansion yields
\[
 \begin{aligned}
 (n-1)\int_{B_{s_1}}\chi w^2
 ={}&\int_{B_{s_1}}\bigl((x\cdot p)(p\cdot D\chi)-w^2x\cdot D\chi\bigr)\\
 &+\int_{B_{s_1}} (K-K_{B_{s_1}})\operatorname{div}(x\chi),
 \end{aligned}
\]
where $K_{B_{s_1}}$ denotes the integral mean of $K$ over $B_{s_1}$.

The first integral on the right is bounded by $C(n,r)T$ because $D\chi$ is supported outside $B_r$. For the second integral, by Poincar\'e's inequality, \eqref{eq:stress}, H\"older's inequality, and the estimate of $\mathcal A_u$ above, we get
\[
 \|K-K_{B_{s_1}}\|_{L^1(B_{s_1})}
 \le C(n,r)\|DK\|_{L^1(B_{s_1})}
 \le C(n,r)(S+T)^{1/2}T^{1/2}.
\]
It follows that $(n-1)S\le C(n,r)T+C(n,r)(S+T)^{1/2}T^{1/2}$. From Young's inequality we obtain \eqref{eq:annulus-energy}. The estimate \eqref{eq:annulus-data} follows from \eqref{eq:annulus-energy} and \eqref{eq:endpoint-estimate}, together with a standard scaling and covering argument.
\end{proof}

The following interpolation result will be used in the proof of Theorem \ref{thm:gradient} to deduce gradient estimates.

\begin{lemma}\label{lem:morrey-bv}
Let $V\in W^{1,1}(\RR^n;\RR^m)$ have compact support, let $p>0$, and let $1\le\beta<n$. Then
\begin{equation}\label{eq:morrey-bv}
 \int_{\RR^n}|V|^{1+p/\beta}
 \le C(n)\|V\|_{M^{p,\beta}(\RR^n)}^{p/\beta}
                      \int_{\RR^n}|DV|.
\end{equation}
\end{lemma}

\begin{proof}
By H\"older's inequality, the measure $d\lambda=|V|^{p/\beta}dx$ satisfies
\[
 \lambda(B_\rho(x))\le C(n)\|V\|_{M^{p,\beta}(\RR^n)}^{p/\beta}\rho^{n-1}.
\]
From \cite[Thm.\,4.7]{MeyersZiemer1977} we have
\[
 \int_{\RR^n}|\varphi|\,d\lambda
 \le C(n)\|V\|_{M^{p,\beta}(\RR^n)}^{p/\beta}\int_{\RR^n}|D\varphi|
 \qquad(\;\forall\varphi\in C_c^1(\RR^n)).
\]
By approximation we can take $\varphi=|V|$ and thus obtain \eqref{eq:morrey-bv}.
\end{proof}

The following lemma recalls a classical identity; see, for instance, \cite[Lem.\,A.9]{LiskevichLyakhovaMoroz2006}. We use it to characterize stable solutions in the proof of Theorem~\ref{thm:obstructions}.

\begin{lemma}\label{lem:ground-state}
Let $V\in L^\infty_{\loc}(\Omega)$ and let $\phi\in C^2(\Omega)$ be positive with $-\Delta\phi=V\phi$. Then
\begin{equation}\label{eq:ground-state}
 \int_\Omega\bigl(|D\xi|^2-V\xi^2\bigr)
 =\int_\Omega\phi^2\left|D\left(\frac\xi\phi\right)\right|^2
 \qquad(\xi\in C_c^1(\Omega)).
\end{equation}
In particular, the quadratic form on the left is nonnegative.
\end{lemma}

\subsection{Energy decay estimates}

The following lemma collects several useful energy decay estimates obtained in \cite{CabreFigalliRosOtonSerra2020,Peng2024,PengZhangZhou2024}. Specifically, \eqref{eq:radial-critical} follows from the proof of \cite[Lem.\,2.1]{Peng2024}, while \eqref{eq:radial-budget} and \eqref{eq:radial-high} follow from \cite[Lem.\,2.1]{CabreFigalliRosOtonSerra2020} and \cite[Lem.\,1.7]{PengZhangZhou2024}, respectively.

\begin{lemma}\label{lem:radial}
Let $u$ be a classical stable solution of \eqref{eq:pde} in $B_R(y)$. If $3\le n\le10$, then
\begin{equation}\label{eq:radial-critical}
 r^{2-n}\int_{B_r(y)}|Du|^2
 \le C(n)R^{2-n}\int_{B_R(y)\setminus B_{3R/4}(y)}|Du|^2
 \qquad(0<r\le R/2).
\end{equation}
If $3\le n\le9$, then also
\begin{equation}\label{eq:radial-budget}
 \int_{B_{R/3}(y)}|x-y|^{2-n}u_{r,y}^2
 \le C(n)R^{2-n}\int_{B_R(y)\setminus B_{3R/4}(y)}|Du|^2,
\end{equation}
where $u_{r,y}=(x-y)\cdot Du/|x-y|$. If $n\ge11$, then
\begin{equation}\label{eq:radial-high}
 r^{\beta_n-n}\int_{B_r(y)}|Du|^2
 \le C(n)R^{\beta_n-n}\int_{B_R(y)\setminus B_{3R/4}(y)}|Du|^2
 \qquad(0<r\le R/2).
\end{equation}
\end{lemma}

The following result establishes a key energy decay estimate used in the proofs of Theorems~\ref{thm:gradient} and~\ref{thm:holder} in dimensions $2\le n\le5$. The key idea for $3\le n\le5$ is to combine the equations satisfied by the rotational derivatives $W_{ij}=x_i\partial_j u-x_j\partial_i u$ with the stability inequality \eqref{eq:stability}. The case $n=2$ follows by a dimension-lifting argument.

\begin{lemma}\label{lem:rotation}
For $2\le n\le5$, every classical stable solution of \eqref{eq:pde} in $B_R(y)$ satisfies
\begin{equation}\label{eq:low-energy}
 \int_{B_r(y)}|Du|^2
 \le C(n)(r/R)^{n-2+2\alpha_n}\int_{B_R(y)}|Du|^2
 \qquad(0<r\le R/2),
\end{equation}
where $\alpha_n$ is defined in \eqref{eq:alpha}.
\end{lemma}
\begin{proof}
First let $3\le n\le5$ and assume $y=0$ for simplicity. For $i<j$ we set
$W_{ij}=x_i\partial_j u-x_j\partial_i u$ and $W=(W_{ij})_{i<j}\in\RR^{\binom n2}$, with
$|W|^2=\sum_{i<j}W_{ij}^2$ and $|DW|^2=\sum_{k=1}^n\sum_{i<j}|\partial_kW_{ij}|^2$.
Since rotations commute with the Laplacian, from \eqref{eq:jacobi} we can get
$(-\Delta-f'_-(u))W_{ij}=0$ for every $i<j$.
For $x\ne0$, we define the tangential gradient
\[
 D_\tau u=Du-\frac{x\cdot Du}{|x|^2}x.
\]
The identity
\[
 \sum_{i<j}(x_i\partial_j u-x_j\partial_i u)^2
 =|x|^2|Du|^2-(x\cdot Du)^2
 =|x|^2|D_\tau u|^2
\]
shows that, with $\rho=|W|$, one has $\rho=|x||D_\tau u|$.
On $\{\rho>0\}$, we regard $W/\rho$ as an $\RR^{\binom n2}$-valued map and use the notation
\[
 \left|D\left(\frac W\rho\right)\right|^2
 :=\sum_{k=1}^n\sum_{i<j}
 \left|\partial_k\left(\frac{W_{ij}}\rho\right)\right|^2.
\]

Take $\eta\in C_c^{0,1}(B_R)$. Testing $(-\Delta-f'_-(u))W_{ij}=0$ with $W_{ij}\eta^2$, summing over $i<j$, and integrating by parts gives
\[
 \int_{B_R} \eta^2|DW|^2+2\int_{B_R} \eta\rho D\rho\cdot D\eta
       -\int_{B_R} f'_-(u)\rho^2\eta^2=0.
\]
On the other hand, taking $\xi=\rho\eta$ in \eqref{eq:stability} gives
\[
 0\leq Q_u(\rho\eta)=\int_{B_R} \eta^2|D\rho|^2
 +2\int_{B_R} \eta\rho D\rho\cdot D\eta
 +\int_{B_R} \rho^2|D\eta|^2-\int_{B_R} f'_-(u)\rho^2\eta^2.
\]
Since $DW=0$ almost everywhere on $\{\rho=0\}$, and on $\{\rho>0\}$ one has
$|DW|^2-|D\rho|^2=\rho^2|D(W/\rho)|^2$, from the above two estimates we have
\begin{equation}\label{eq:vector-rotation}
 \int_{B_R\cap\{\rho>0\}}\rho^2
      \left|D\left(\frac W\rho\right)\right|^2\eta^2
 \le\int_{B_R}\rho^2|D\eta|^2.
\end{equation}

Next we estimate the derivative of $W/\rho$. On $\{\rho>0\}$, we put $e=x/|x|$ and
$\tau=D_\tau u/|D_\tau u|$. Then $e\cdot\tau=0$ and
$W_{ij}/\rho=e_i\tau_j-e_j\tau_i$. Fix $k$ and write $a=\partial_ke$ and $b=\partial_k\tau$. Since $|e|=|\tau|=1$ and $e\cdot\tau=0$, if $\lambda=a\cdot\tau=-e\cdot b$, then
$a=\lambda\tau+a^\perp$ and $b=-\lambda e+b^\perp$, where $a^\perp,b^\perp$ are orthogonal to both $e$ and $\tau$. Hence
\[
 \partial_k\left(\frac{W_{ij}}\rho\right)
 =a_i^\perp\tau_j-a_j^\perp\tau_i
   +e_i b_j^\perp-e_j b_i^\perp.
\]
Using the elementary identity
$\sum_{i<j}(v_iw_j-v_jw_i)^2=|v|^2|w|^2-(v\cdot w)^2$, and the above orthogonality, we obtain
\[
 \sum_{i<j}\left|\partial_k\left(\frac{W_{ij}}\rho\right)\right|^2
 =|a^\perp|^2+|b^\perp|^2\ge |a^\perp|^2.
\]
Since
$|a^\perp|^2=|\partial_ke|^2-|\partial_ke\cdot\tau|^2$,
summing in $k$ gives
\[
 \left|D\left(\frac W\rho\right)\right|^2
 \ge\frac{n-2}{|x|^2}.
\]
Together with \eqref{eq:vector-rotation} and $\rho=|x||D_\tau u|$, this yields
\begin{equation}\label{eq:rotation-test}
 (n-2)\int_{B_R}|D_\tau u|^2\eta^2
 \le\int_{B_R}|x|^2|D_\tau u|^2|D\eta|^2.
\end{equation}
By taking $\eta=\max\{r,|x|\}^{-\sqrt{n-2}}\chi$, where $\chi=1$ on $B_{3R/4}$, is compactly supported in $B_R$, and satisfies $|D\chi|\le C/R$, we obtain from \eqref{eq:rotation-test} that
\begin{equation}\label{eq:angular-decay}
 \int_{B_r}|D_\tau u|^2
 \le C(n)(r/R)^{2\sqrt{n-2}}\int_{B_R}|Du|^2.
\end{equation}

To recover the full gradient on the left, we just need to average \eqref{eq:angular-decay} over different centers, as in \cite[Proof of Thm.\,1.2, Step 2]{CabreFigalliRosOtonSerra2020}. We omit the details. Finally, to treat the case $n=2$, one just needs to apply a standard dimension lifting argument.
\end{proof}

The following lemma constructs vector fields satisfying the divergence identities \eqref{eq:reaction-div} and the $L^1$ bounds \eqref{eq:reaction-size}. The main idea is to first establish the general estimate \eqref{eq:convex-form} for a suitable class of convex functions $\Phi$. By choosing $\Phi$ appropriately and combining this estimate with the stability inequality \eqref{eq:stability}, we obtain bounded linear functionals on a suitable Hilbert space. The Riesz representation theorem then yields the desired vector fields.

\begin{lemma}\label{lem:reaction}
Let $\Omega_1\Subset\Omega$, let $\eta\in C_c^\infty(\Omega)$ be equal to one on a neighborhood of $\ol{\Omega_1}$, and let $u$ be a classical stable solution of \eqref{eq:pde} in $\Omega$. Set $z=x\cdot Du$ and $q=Q_u(z\eta)$. Then there are vector fields $W_i,W_{ik}\in L^1(\Omega_1;\RR^n)$ such that
\begin{equation}\label{eq:reaction-div}
 f(u)\partial_i u=\operatorname{div}W_i,
 \qquad f(u)\frac{\partial_i u\partial_k u}{|Du|}
       =\operatorname{div}W_{ik}
 \quad\hbox{in }\mathcal D'(\Omega_1),
\end{equation}
where the quotient is zero on $\{x\in\Omega_1:Du(x)=0\}$, and
\begin{equation}\label{eq:reaction-size}
 \|W_i\|_{L^1(\Omega_1)}\le\tfrac12\sqrt q\,\|Du\|_{L^2(\Omega_1)},
 \qquad
 \|W_{ik}\|_{L^1(\Omega_1)}\le\tfrac72\sqrt q\,\|Du\|_{L^2(\Omega_1)}.
\end{equation}
Here for a vector field $W\in L^1(\Omega_1;\RR^n)$, 
$\|W\|_{L^1(\Omega_1)}:=\int_{\Omega_1}|W(x)|\,dx$.
\end{lemma}
\begin{proof}
First we claim that, if $\Phi:\RR^n\to[0,\infty)$ is convex, Lipschitz, and positively homogeneous of degree one, then
\begin{equation}\label{eq:convex-form}
 0\le Q_u(\Phi(Du)\psi)
 \le\int_{\Omega_1}\Phi(Du)^2|D\psi|^2
 \qquad(\forall\;\psi\in C_c^1(\Omega_1)).
\end{equation}
The first inequality follows from the stability \eqref{eq:stability}. To prove the second inequality, first we suppose that $\Phi$ is smooth and put $V=f'_-(u)$. By \eqref{eq:jacobi}, $(-\Delta-V)\partial_i u=0$ in $\mathcal D'(\Omega_1)$. Since Euler's identity gives $D\Phi(p)\cdot p=\Phi(p)\;(p\in\RR^n)$, the chain rule and convexity yield
\[
 \begin{aligned}
 (-\Delta-V)\Phi(Du)
 &=-\sum_i\Phi_i(Du)\Delta(\partial_i u)
   -\sum_{\alpha=1}^nD^2\Phi(Du)
       [\partial_\alpha Du,\partial_\alpha Du]-V\Phi(Du)\\
 &=-\sum_{\alpha=1}^nD^2\Phi(Du)
       [\partial_\alpha Du,\partial_\alpha Du]\le0
 \end{aligned}
\]
in $\Omega_1$. Testing this inequality in $\Omega_1$ with $\Phi(Du)\psi^2$ gives
\[
 \int_{\Omega_1}\psi^2|D\Phi(Du)|^2
 +2\int_{\Omega_1}\Phi(Du)\psi D\Phi(Du)\cdot D\psi
 -\int_{\Omega_1}V\Phi(Du)^2\psi^2\le0.
\]
Expanding $Q_u(\Phi(Du)\psi)$ and using this inequality proves \eqref{eq:convex-form}.

For general $\Phi$, set $\Phi_\varepsilon=\Phi*\rho_\varepsilon$  with a standard mollifier $\rho_\varepsilon$ on $\RR^n$. Convexity is preserved, and Euler's identity for $\Phi$ gives
$|D\Phi_\varepsilon(p)\cdot p-\Phi_\varepsilon(p)|\le C\operatorname{Lip}(\Phi)\varepsilon$. Since $V$ is bounded on the support of $\psi$, the preceding calculation has an error tending to zero with $\varepsilon$. Therefore, a standard convergence argument yields \eqref{eq:convex-form}.

Set $\sigma_i(p)=p_i$, and $\sigma_{ik}(p)=\frac{p_ip_k}{|p|},\;\sigma_{ik}(0)=0$. The functions $\Phi_1(p):=\sigma_{ik}(p)+3|p|$ and $\Phi_2(p):=3|p|$ are nonnegative, convex, and homogeneous of degree one. Therefore, we can apply \eqref{eq:convex-form} to $\Phi_1$ and $\Phi_2$ to obtain
\[
 Q_u(\Phi_1(Du)\psi)^{1/2}\le4\||Du||D\psi|\|_{L^2(\Omega_1)},
 \qquad
 Q_u(\Phi_2(Du)\psi)^{1/2}\le3\||Du||D\psi|\|_{L^2(\Omega_1)}.
\]
Since $Q_u$ is nonnegative on $H_0^1(\Omega_1)$, the triangle inequality gives
\[
 \begin{aligned}
 Q_u(\sigma_{ik}(Du)\psi)^{1/2}
 &\le Q_u(\Phi_1(Du)\psi)^{1/2}
      +Q_u(\Phi_2(Du)\psi)^{1/2}\\
 &\le7\||Du||D\psi|\|_{L^2(\Omega_1)}.
 \end{aligned}
\]
For $\sigma_i$, by definition we have $Q_u(\sigma_{i}(Du)\psi)^{1/2}\leq\||Du||D\psi|\|_{L^2(\Omega_1)}$.

For either $\sigma=\sigma_i$ or $\sigma_{ik}$, we can obtain from \eqref{eq:jacobi} that
\[
 2\int_{\Omega_1}f(u)\sigma(Du)\psi
 =Q_u(z\eta,\sigma(Du)\psi).
\]
By the Cauchy--Schwarz inequality for the nonnegative form $Q_u$ we deduce
\[
 \left|\int_{\Omega_1}f(u)\sigma(Du)\psi\right|
 \le\frac{c_\sigma}{2}\sqrt q\,
       \||Du||D\psi|\|_{L^2(\Omega_1)},
 \qquad c_{\sigma_i}:=1,\quad c_{\sigma_{ik}}:=7.
\]

Set
\[
 \mathcal H=\overline{\{|Du|D\psi:\psi\in C_c^1(\Omega_1)\}}^{\,L^2(\Omega_1;\RR^n)}
\]
and define $\Lambda_\sigma(|Du|D\psi)=\int_{\Omega_1}f(u)\sigma(Du)\psi$. The preceding bound shows that $\Lambda_\sigma$ is well defined as a functional of $|Du|D\psi$, and it extends uniquely to a bounded linear functional on $\mathcal H$, with $\|\Lambda_\sigma\|_{\mathcal H^*}\le c_\sigma\sqrt q/2$. By the Riesz representation theorem, there is $Z_\sigma\in\mathcal H\subseteq L^2(\Omega_1;\RR^n)$ such that
\[
 \Lambda_\sigma(Y)=\int_{\Omega_1}Z_\sigma\cdot Y
 \quad(\forall\;Y\in\mathcal H),
 \qquad
 \|Z_\sigma\|_{L^2(\Omega_1)}\le\frac{c_\sigma}{2}\sqrt q.
\]
Taking $Y=|Du|D\psi$ and setting $W_\sigma=-|Du|Z_\sigma$ gives
\[
 \int_{\Omega_1}f(u)\sigma(Du)\psi
 =-\int_{\Omega_1}W_\sigma\cdot D\psi
 =\langle\operatorname{div}W_\sigma,\psi\rangle;
\]
thus $f(u)\sigma(Du)=\operatorname{div}W_\sigma$ in $\mathcal D'(\Omega_1)$. Moreover, by H\"older's inequality we have
\[
 \|W_\sigma\|_{L^1(\Omega_1)}
 \le\|Du\|_{L^2(\Omega_1)}\|Z_\sigma\|_{L^2(\Omega_1)}
 \le\frac{c_\sigma}{2}\sqrt q\,\|Du\|_{L^2(\Omega_1)}.
\]
Taking $W_i=W_{\sigma_i}$ and $W_{ik}=W_{\sigma_{ik}}$ proves \eqref{eq:reaction-div}--\eqref{eq:reaction-size}.
\end{proof}

We next establish a quantitative result showing that, if a function $h$ is bounded above by some $k\in W^{1,1}$ and satisfies $Dh=\Div W$ for a matrix field $W$ with small $L^1$ norm, then $h$ is close to a constant on a smaller domain, provided suitable integral bounds hold. The smallness of $W$ alone is insufficient to guarantee this conclusion. Instead, we first apply the classical weak $(1,1)$ estimate for singular integrals to derive the superlevel-set bound \eqref{add4}. The inequality $h\le k$, together with the Sobolev embedding $W^{1,1}\hookrightarrow L^{n/(n-1)}$, then yields the improved estimate \eqref{eq:primitive-three-bounds}. Finally, the layer-cake formula gives the desired $L^1$ bound for $h-b$.

\begin{lemma}\label{lem:primitive}
Fix $n\ge2$ and let $A_j=B_{2^j}\setminus\ol{B_{2^{-j}}}$. Suppose that
$h\in L^1(A_5)$, $k\in W^{1,1}(A_5)$, and
$W\in L^1(A_5;\RR^{n\times n})$ satisfy $h\le k$ almost everywhere in $A_5$ and
\[
 Dh=\Div W\quad\hbox{in }\mathcal D'(A_5),\qquad
 \|h\|_{L^1(A_5)}+\|k\|_{W^{1,1}(A_5)}\le M,
 \qquad \|W\|_{L^1(A_5)}\le\varepsilon.
\]
Then there exists $b\in\RR$, $|b|\le C(n)M$, such that
\begin{equation}\label{eq:primitive}
 \|h-b\|_{L^1(A_4)}\le C(n)M^{1/2}\varepsilon^{1/2}.
\end{equation}
\end{lemma}
\begin{proof}
The cases $M=0$ and $\varepsilon\ge M$ are immediate. If $\varepsilon=0$, then $W=0$ and hence $Dh=0$ in $\mathcal D'(A_5)$. Since $A_5$ is connected, $h$ is constant almost everywhere, and the conclusion follows. Therefore, by normalization, we may assume $M=1$ and $0<\varepsilon<1$ in the following.

Set
\[
 U_1=B_{20}\setminus\ol{B_{1/20}},\qquad
 U_2=B_{24}\setminus\ol{B_{1/24}},
\]
so that $\ol{A_4}\Subset U_1\Subset U_2\Subset A_5$. Choose
$\chi\in C_c^\infty(A_5)$ with $\chi=1$ on a neighborhood of $\ol{U_2}$, and extend $\chi W$ by zero from $A_5$ to $\RR^n$. Let $\Gamma$ be the fundamental solution such that $\Delta\Gamma=\delta_0$. Consider the distribution
\begin{equation}\label{eq:PW-def}
 P_W:=\sum_{i,j=1}^n\partial_i\partial_j\Gamma*(\chi W_{ij})
 \qquad\hbox{in }\mathcal D'(\RR^n).
\end{equation}
Explicitly, for $\varphi\in C_c^\infty(\RR^n)$,
\[
 \langle P_W,\varphi\rangle
 =\sum_{i,j=1}^n\int_{A_5}\chi(y)W_{ij}(y)
       (\partial_i\partial_j\Gamma*\varphi)(y)\,dy.
\]
Since $\chi=1$ on $U_2$, \eqref{eq:PW-def} and $Dh=\Div W$ give, in $\mathcal D'(U_2)$,
\[
 \Delta P_W=\sum_{i,j}\partial_i\partial_jW_{ij}
 =\Delta h.
\]
Hence $H:=h-P_W$ satisfies $\Delta H=0$ in $\mathcal D'(U_2)$. By Weyl's lemma, $H$ has a smooth harmonic representative on $U_2$, and consequently $P_W=h-H\in L^1_{\loc}(U_2)$.

For $1\le i,j\le n$, let
\[
 T_{ij}F:=\partial_i\partial_j\Gamma*F
 \qquad\hbox{in }\mathcal D'(\RR^n).
\]
The distribution kernel has the standard Calder\'on--Zygmund decomposition
\[
 \partial_i\partial_j\Gamma
 =\operatorname{p.v.}K_{ij}+c_n\delta_{ij}\delta_0,
\]
where $K_{ij}$ is homogeneous of degree $-n$, smooth on $\RR^n\setminus\{0\}$, and has zero spherical mean. Hence the classical weak $(1,1)$ estimate for singular integrals (see, for example,
\cite[Chap.\,II, Sec.\,1]{Stein1970}), together with Chebyshev's inequality for the diagonal local term, gives
\[
 |\{x\in\RR^n:|T_{ij}F(x)|>t\}|
 \le \frac{C(n)}{t}\|F\|_{L^1(\RR^n)}
 \qquad(F\in L^1(\RR^n),\ t>0).
\]
Applying this to $F_{ij}=\chi W_{ij}$ yields
\begin{equation}\label{eq:primitive-weak}
 |\{x\in U_1:|P_W(x)|>t\}|
 \le \frac{C(n)\varepsilon}{t}
 \qquad(t>0).
\end{equation}

Choose $r_0>0$ small such that $U_1+B_{r_0}\Subset U_2$. Define
\[
 \vartheta_0(x)=c_n
 \begin{cases}
 \exp\!\left(-\dfrac1{1-|x|^2}\right),&|x|<1,\\[4pt]
 0,&|x|\ge1,
 \end{cases}
 \qquad \int_{\RR^n}\vartheta_0=1,
\]
and set $\vartheta(x)=r_0^{-n}\vartheta_0(x/r_0)$. Since $H$ is harmonic in $U_2$, the mean-value property gives $H=H*\vartheta$ on $U_1$. As $Dh=\Div W$, on $U_1$ we have that
\[
 DH=D(h*\vartheta)-D(P_W*\vartheta)
    =(\Div W)*\vartheta-D(P_W*\vartheta).
\]
For the first term on the right, componentwise,
\[
 [((\Div W)*\vartheta)(x)]_i
 =\sum_{j=1}^n\int_{A_5}W_{ij}(y)\,\partial_j\vartheta(x-y)\,dy,
 \qquad x\in U_1.
\]
Thus we have
\[
 \|(\Div W)*\vartheta\|_{L^\infty(U_1)}
 \le C(n)\|D\vartheta\|_{L^\infty(\RR^n)}\|W\|_{L^1(A_5)}
 \le C(n)\varepsilon.
\]
For the second term, we can write
\[
 D(P_W*\vartheta)
 =\sum_{i,j}\bigl(D\partial_i\partial_j\Gamma*\vartheta\bigr)
       *(\chi W_{ij}).
\]
As each $D\partial_i\partial_j\Gamma*\vartheta$ belongs to $C^\infty(\RR^n)$, we have
\[
 \|D(P_W*\vartheta)\|_{L^\infty(U_1)}
 \le C(n)\|W\|_{L^1(A_5)}\le C(n)\varepsilon.
\]
Consequently, we obtain
\begin{equation}\label{eq:H-osc}
 \|DH\|_{L^\infty(U_1)}\le C(n)\varepsilon,
 \qquad \operatorname*{osc}_{U_1}H\le C(n)\varepsilon.
\end{equation}

Choose $\zeta\in C_c^\infty(U_1)$ with $\zeta\ge0$, $\int_{U_1}\zeta=1$, and $\zeta\ge c(n)>0$ on $A_4$, and set
$b=\int_{U_1}h\zeta$. Then $|b|\le C(n)$. Moreover,
\[
 |\langle P_W,\zeta\rangle|
 \le C(n)\|W\|_{L^1(A_5)}\le C(n)\varepsilon,
\]
because every $\partial_i\partial_j\Gamma*\zeta$ belongs to $C^\infty(\RR^n)$. Set $H_\zeta=\int_{U_1}H\zeta$; then
$b-H_\zeta=\langle P_W,\zeta\rangle$, while \eqref{eq:H-osc} gives
$|H-H_\zeta|\le C(n)\varepsilon$ on $U_1$. Hence
$|H-b|\le C(n)\varepsilon$ on $U_1$. Combining this with \eqref{eq:primitive-weak} gives
\begin{align}\label{add4}
    |\{x\in U_1:|h(x)-b|>t\}|
 \le \min\{|U_1|,C(n)\varepsilon/t\}
 \qquad(t>0).
\end{align}

Since $h\le k$ in $A_5$, one has $(h-b)_+\le(k-b)_+$ on $U_1$. The $W^{1,1}$-Sobolev inequality  together with
$\|k\|_{W^{1,1}(A_5)}+|b|\le C(n)$ gives
\[
 \|(h-b)_+\|_{L^{n/(n-1)}(U_1)}\le C(n).
\]
From Chebyshev's inequality and \eqref{add4} we deduce that
\begin{equation}\label{eq:primitive-three-bounds}
 |\{x\in U_1:(h(x)-b)_+>t\}|
 \le\min\{|U_1|,C(n)\varepsilon/t,C(n)t^{-n/(n-1)}\}.
\end{equation}
By the layer-cake formula we get
\[
 \begin{aligned}
 \int_{U_1}(h-b)_+
 &=\int_0^\infty |\{x\in U_1:(h(x)-b)_+>t\}|\,dt\\
 &\le |U_1|\varepsilon
   +C(n)\varepsilon\int_\varepsilon^{\varepsilon^{-(n-1)}}\frac{dt}{t}
   +C(n)\int_{\varepsilon^{-(n-1)}}^\infty t^{-n/(n-1)}\,dt\\
 &\le C(n)\varepsilon(1+|\log\varepsilon|)
 \le C(n)\sqrt\varepsilon.
 \end{aligned}
\]

Finally, since $\int_{U_1}\zeta(h-b)=0$ and $\zeta\ge c(n)$ on $A_4$, we see that
\[
 \begin{aligned}
 \int_{A_4}(h-b)_-
 &\le C(n)\int_{U_1}\zeta(h-b)_-
 =C(n)\int_{U_1}\zeta(h-b)_+\\
 &\le C(n)\|\zeta\|_{L^\infty(U_1)}
       \int_{U_1}(h-b)_+
 \le C(n)\sqrt\varepsilon.
 \end{aligned}
\]
Combining the above two estimates proves \eqref{eq:primitive}.
\end{proof}

Proposition~\ref{prop:rigidity} is the main technical ingredient in establishing the energy decay estimate \eqref{eq:energy-decay} for $6\le n\le9$. Roughly speaking, it shows that if the gradient energy of a stable solution $u$ is bounded on an outer annulus, while the radial derivative $x\cdot Du$ and the associated stability quadratic form are sufficiently small, then the gradient energy on an inner annulus must also be small.

Since the proof is rather involved, we briefly outline its main steps. First, the first identity in \eqref{eq:reaction-div}, together with Lemma~\ref{lem:primitive}, shows that $F(u)$ is close to a constant $b$ in $L^1$, where $F'=f$. This leads to \eqref{eq:small-set-energy}, which provides a bound for the gradient energy over sets of small measure. Consequently, by choosing $M$ sufficiently large, we can reduce the problem to estimating the energy over $\{|u-c|\le M\}$, where $c$ is the average of $u$ over a suitable annulus. Using two cutoff functions, one in the spatial variable and the other in the range of $u-c$, we further reduce the problem to estimating an integral of the form $\int Du\cdot Dg$, where $g$ satisfies \eqref{eq:bounded-test}; see \eqref{add5.5}.

To estimate this integral, we proceed in three steps. First, we consider the vector field $V=|Du|Du$. The second identity in \eqref{eq:reaction-div} yields the decomposition \eqref{eq:V-equation} for $DV$. Using this decomposition, the weak $(1,1)$ estimate for singular integrals, and the preceding energy bounds, we establish the $L^1$ translation estimate \eqref{eq:V-translation}. In particular, this allows us to approximate $V$ in $L^1$ by its mollification $V_\sigma$.

Next, we construct a smooth map $H:\RR^n\to\RR^n$ that approximates the inverse of the map $p\mapsto|p|p$. More precisely, $H(V)$ approximates $Du$ pointwise; see \eqref{eq:H-composition}. Combining this with the estimates for $V_\sigma-V$, we show that $H(V_\sigma)$ approximates $Du$ in $L^2$; see \eqref{eq:inverse-approximation}. Hence, the integral $\int Du\cdot Dg$ can be approximated by
\[
 \int H(V_\sigma)\cdot Dg
 =-\int g\,\operatorname{div}H(V_\sigma).
\]

Finally, we expand $\operatorname{div}H(V_\sigma)$ as in \eqref{eq:approx-divergence}. The principal term is small by \eqref{eq:trace-small}, while the remaining terms are controlled by the commutator estimate \eqref{eq:commutator}. Combining these estimates and choosing the parameters in a suitable order, we obtain the desired smallness of the gradient energy on the inner annulus, completing the proof.

\begin{proposition}\label{prop:rigidity}
Fix $6\le n\le9$, let $A_j=B_{2^j}\setminus\ol{B_{2^{-j}}}$, and fix a radial function $\eta\in C_c^\infty(A_6)$ equal to one near $\ol{A_5}$. Then for every $\varepsilon_*\in(0,1)$ there exists $0<\delta_*=\delta_*(n,\varepsilon_*)<1$ such that every classical stable solution $u$ of \eqref{eq:pde} in $A_6$ with
\begin{equation*}
 \int_{A_6}|Du|^2\le1,\qquad
 \|x\cdot Du\|_{L^2(A_6)}+Q_u((x\cdot Du)\eta)^{1/2}\le\delta_*
\end{equation*}
satisfies $\int_{A_1}|Du|^2\le\varepsilon_*$.
\end{proposition}
\begin{proof}
For simplicity, we write $p=Du$, $w=|p|$, $z=x\cdot Du$, $q=Q_u(z\eta)$, $d=\|z\|_{L^2(A_6)}$ and $B=D^2u+f(u)\frac{p\otimes p}{w^2}$. By \eqref{eq:matrix} and \eqref{eq:geometric}, we have $\|B\|_{L^2(A_5)}\le C(n)$. Take $F$ such that $F'=f$ and $K:=w^2/2+F(u)$ has mean zero on $A_5$. Then from \eqref{eq:stress} and Poincar\'e's inequality we get
\[
 \|K\|_{W^{1,1}(A_5)}+\|F(u)\|_{L^1(A_5)}\le C(n).
\]
The first assertion in \eqref{eq:reaction-div} gives
$D(F(u))=\Div E^\pi$ in $\mathcal D'(A_5)$, with $\|E^\pi\|_{L^1(A_5)}\le C\sqrt q$. Since $F(u)\le K$, from Lemma~\ref{lem:primitive} we can find $b\in\RR$ such that $|b|\le C$ and $\|F(u)-b\|_{L^1(A_4)}\le Cq^{1/4}$. Set
\[
 k=2|K-b|,\qquad e_0=2|F(u)-b|.
\]
Then
\begin{equation}\label{eq:energy-majorant}
 w^2\le k+e_0,\qquad
 \|k\|_{L^{n/(n-1)}(A_4)}\le C,\qquad
 \|e_0\|_{L^1(A_4)}\le Cq^{1/4}.
\end{equation}
In particular, for any measurable set $E\subseteq A_4$, we have by H\"older's inequality that
\begin{equation}\label{eq:small-set-energy}
 \int_E w^2\le C|E|^{1/n}+Cq^{1/4}.
\end{equation}

Next we introduce the map
\[
 \mathcal P:\RR^n\to\RR^n,\qquad \mathcal P(\xi)=|\xi|\xi,
\]
and define the vector field
\[
 V(x):=\mathcal P(p(x))=|Du(x)|Du(x),\qquad x\in A_5.
\]
The map $\mathcal P$ is $C^{1,1}$, $D\mathcal P$ is globally Lipschitz, and
\[
 D\mathcal P(\xi)=|\xi|I+\frac{\xi\otimes\xi}{|\xi|}\quad(\xi\ne0),
 \qquad D\mathcal P(0)=0.
\]
Let $W_{ij}$ be the vector field supplied by Lemma~\ref{lem:reaction} and set $E=(E_{ijk})$ by
$E_{ijk}=-2(W_{ij})_k$; then by the chain rule and the second assertion of \eqref{eq:reaction-div} we obtain
\[
 \partial_jV_i
 =[D\mathcal P(p)B]_{ij}-2f(u)\frac{p_ip_j}{w}=[D\mathcal P(p)B]_{ij}+\sum_{k=1}^n\partial_kE_{ijk}
\]
in $\mathcal D'(A_5)$. Therefore, we have
\begin{equation}\label{eq:V-equation}
 DV=D\mathcal P(p)B+\Div E,\qquad
 \|D\mathcal P(p)B\|_{L^1(A_5)}\le C,\qquad \|E\|_{L^1(A_5)}\le C\sqrt q.
\end{equation}

Fix an open set $U$ with $\ol{A_2}\Subset U\Subset A_3$, and choose $s_0>0$ small such that
$A_2+B_{s_0}\Subset U$ and $U+B_{s_0}\Subset A_3$. We claim that
\begin{equation}\label{eq:V-translation}
 \int_U|V(x+a)-V(x)|\,dx
 \le\Omega(s,q):=C(s^{1/4}+q^{1/4})
 \qquad(\forall\;|a|\le s<s_0).
\end{equation}
To prove it, we choose $\xi\in C_c^\infty(A_4)$ equal to one on a neighborhood of $\ol{A_3}$. For a fixed component $i$, set $g=\xi V_i$ in $A_4$ and extend it by zero to $\RR^n$. Write
$A=D\mathcal P(p)B$. Then \eqref{eq:V-equation} means
\[
 \partial_jV_i=A_{ij}+\sum_{k=1}^n\partial_kE_{ijk}
 \qquad\hbox{in }\mathcal D'(A_4).
\]
For each $j$, one can compute
\[
 \begin{aligned}
 \partial_jg
 =\xi\,\partial_jV_i+V_i\partial_j\xi
 =\xi A_{ij}+V_i\partial_j\xi
   -\sum_{k=1}^nE_{ijk}\partial_k\xi
   +\sum_{k=1}^n\partial_k(\xi E_{ijk}).
 \end{aligned}
\]
Thus, with
\[
 \mathcal E_{jk}:=\xi E_{ijk},\qquad
 \lambda_j:=\xi A_{ij}+V_i\partial_j\xi
             -\sum_{k=1}^nE_{ijk}\partial_k\xi,
\]
we have
\begin{equation}\label{eq:g-decomposition}
 Dg=\lambda+\Div\mathcal E
 \qquad\hbox{in }\mathcal D'(\RR^n),
\end{equation}
and
\[
 \|\lambda\|_{L^1(\RR^n)}\le C,
 \qquad
 \|\mathcal E\|_{L^1(\RR^n)}\le C\sqrt q.
\]

Taking one more divergence in \eqref{eq:g-decomposition} gives
\[
 \Delta g=\operatorname{div}\lambda
 +\sum_{j,k=1}^n\partial_j\partial_k\mathcal E_{jk}
 \qquad\hbox{in }\mathcal D'(\RR^n).
\]
Hence, if we set
\[
 G:=\sum_j\partial_j\Gamma*\lambda_j,
 \qquad
 T:=\sum_{j,k}\partial_j\partial_k\Gamma*\mathcal E_{jk},
\]
where $\Gamma$ is the fundamental solution,
then $g-G-T$ is harmonic on all of $\RR^n$. Since $g$, $\lambda$, and $\mathcal E$ are compactly supported, we must have
\begin{equation}\label{eq:g-newton}
 g=G+T\qquad\hbox{in }\mathcal D'(\RR^n).
\end{equation}
Since $G\in L^1_{\loc}(\RR^n)$ and $g\in L^1(\RR^n)$, this identity shows that $T=g-G\in L^1_{\loc}(\RR^n)$.

We first treat $G$. Let $0<|a|<s_0$ and let $R>1$ be fixed so that $x-y\in B_R$ whenever $x\in U+B_{s_0}$ and $y\in\supp\lambda$. By Fubini's theorem we get
\[
 \begin{aligned}
 \|G(\cdot+a)-G\|_{L^1(U)}
 &\le \|\lambda\|_{L^1(\RR^n)}\int_{B_R}|D\Gamma(z+a)-D\Gamma(z)|\,dz.
 \end{aligned}
\]
Split the last integral into $|z|\le2|a|$ and $2|a|<|z|\le R$. In the first region we use
$|D\Gamma(z)|\le C|z|^{1-n}$, while in the second region we use
$|D^2\Gamma(\zeta)|\le C|\zeta|^{-n}$. It is not hard to compute that
\begin{equation}\label{eq:G-translation}
 \|G(\cdot+a)-G\|_{L^1(U)}
 \le C|a|\left(1+\int_{2|a|}^R\frac{dr}{r}\right)
 \le C|a|\log\frac e{|a|}.
\end{equation}

As for $T$, we can use the Calder\'on--Zygmund weak $(1,1)$ estimate as in Lemma~\ref{lem:primitive} to deduce that, for every $t>0$,
\[
 |\{x\in\RR^n:|T(x)|>t\}|
 \le \frac{C(n)}t\|\mathcal E\|_{L^1(\RR^n)}
 \le \frac{C(n)\sqrt q}{t}.
\]
Thus, for $|a|<s_0$ we have
\begin{equation}\label{eq:T-translation}
 |\{x\in U:|T(x+a)-T(x)|>t\}|
 \le\frac{C(n)\sqrt q}{t}.
\end{equation}
Since
$V(x+a)-V(x)=g(x+a)-g(x)$ for $x\in U$, we can combine Chebyshev's inequality with \eqref{eq:G-translation} and use \eqref{eq:T-translation} to obtain
\begin{equation}\label{eq:X-weak}
 |\{x\in U:|V(x+a)-V(x)|>t\}|
 \le\frac{C(n)(s\log(e/s)+\sqrt q)}t,
 \qquad |a|\le s<s_0.
\end{equation}

To pass from \eqref{eq:X-weak} to \eqref{eq:V-translation}, we set, for $x\in U$,
\[
 X(x)=|V(x+a)-V(x)|,\qquad
 k_a(x)=k(x+a)+k(x),\qquad
 e_a(x)=e_0(x+a)+e_0(x),
\]
and put $Y=(X-e_a)_+$. Since $|V|=w^2$, \eqref{eq:energy-majorant} gives
$X\le k_a+e_a$; hence $0\le Y\le k_a$ and $X\le Y+e_a$. Let
\[
 \Lambda=C\bigl(s\log(e/s)+\sqrt q\bigr).
\]
Note that there are three bounds for the distribution function of $Y$. First, clearly we have
$|\{x\in U:Y(x)>t\}|\le|U|$. Second, as $Y\le X$, from \eqref{eq:X-weak} we have 
$|\{x\in U:Y(x)>t\}|\le\Lambda/t$. Third, from $Y\le k_a$ and \eqref{eq:energy-majorant}, we have
\[
  \|Y\|_{L^{n/(n-1)}(U)}\leq\|k_a\|_{L^{n/(n-1)}(U)}
 \le \|k\|_{L^{n/(n-1)}(U+a)}+\|k\|_{L^{n/(n-1)}(U)}\le C,
\]
which, combined with Chebyshev's inequality, gives $|\{x\in U:Y(x)>t\}|\le Ct^{-n/(n-1)}$. Therefore, we conclude that
\begin{equation}\label{eq:Y-three-bounds}
 |\{x\in U:Y(x)>t\}|
 \le\min\{|U|,\Lambda/t,Ct^{-n/(n-1)}\}.
\end{equation}
Similar to the treatment of \eqref{eq:primitive-three-bounds} in Lemma~\ref{lem:primitive}, we can obtain that
\[
 \int_UY\le C\Lambda(1+|\log\Lambda|)\le C\sqrt\Lambda.
\]
Combining this estimate and \eqref{eq:energy-majorant} yields
\[
 \int_U X\le C\sqrt\Lambda+Cq^{1/4}
 \le C(s^{1/4}+q^{1/4}),
\]
which is exactly \eqref{eq:V-translation}.

Let $\varphi\in C_c^\infty(B_1)$ be fixed, nonnegative, with $\int_{\RR^n}\varphi=1$, and set $\varphi_\sigma(x)=\sigma^{-n}\varphi(x/\sigma)$. For $0<\sigma<s_0/4$ we define
$V_\sigma(x)=\int_{B_\sigma}\varphi_\sigma(a)V(x-a)\,da$ for $x\in A_2$. Since $A_2+B_{s_0}\Subset U\Subset A_3$, from \eqref{eq:V-translation} we can get
\begin{align}
 \|V_\sigma-V\|_{L^1(A_2)}&\le\Omega(\sigma,q),\label{eq:V-mollified}\\
 \int_{A_2}\int_{B_\sigma}\varphi_\sigma(a)|V_\sigma(x)-V(x-a)|\,da\,dx
 &\le\Omega(2\sigma,q).\label{eq:V-double}
\end{align}

Next we choose a fixed smooth nondecreasing function $\Theta$ equal to zero on $(-\infty,0]$ and to one on $[1,\infty)$. For the moment, let $0<\tau<1$, $L\ge1$, and set
\[
 b_0=\tau e^{-L},\qquad
 h(t)=\Theta\left(\frac{\log(t/b_0)}L\right)\ (t>0),\quad h(0)=0,
\]
\[
 H(Z)=h(\sqrt{|Z|})\frac Z{\sqrt{|Z|}}\quad(Z\ne0),\qquad H(0)=0.
\]
Note that $H$ is zero for $|Z|\le b_0^2$. To compute its derivative bounds explicitly, we write
$r=|Z|$, $t=\sqrt r$, and
\[
 H(Z)=a(r)Z,\qquad a(r)=\frac{h(t)}t.
\]
In the case $b_0<t$, we have
\[
 a'(r)=\frac{th'(t)-h(t)}{2t^3},\qquad
 a''(r)=\frac{t^2h''(t)-3th'(t)+3h(t)}{4t^5}.
\]
Since $0\le h\le1$ and $|th'(t)|+|t^2h''(t)|\le C/L$, direct differentiation of
$a(|Z|)Z$ gives
\[
 |DH(Z)|\le C\bigl(|a(r)|+r|a'(r)|\bigr)\le Ct^{-1},
 \qquad
 |D^2H(Z)|\le C\bigl(|a'(r)|+r|a''(r)|\bigr)\le Ct^{-3}.
\]
Since the derivatives of $H$ vanish identically for $t<b_0$, we know that
\begin{equation}\label{eq:H-derivatives}
 \|DH\|_{L^\infty(\RR^n)}\le Cb_0^{-1},\qquad
 \|D^2H\|_{L^\infty(\RR^n)}\le Cb_0^{-3}.
\end{equation}

We consider the composition of $H$ with $\mathcal P$. On $\{w>0\}$,
$|\mathcal P(p)|=w^2$ and therefore
\[
 H(\mathcal P(p))=h(w)p.
\]
Consequently, we have
\begin{align}\label{eq:H-composition}
    |H(\mathcal P(p))-p|=(1-h(w))w\le\tau.
\end{align}
Moreover, direct computation yields
\begin{equation*}
 DH(\mathcal P(p))D\mathcal P(p)=h(w)I+w^{-1}h'(w)p\otimes p.
\end{equation*}
Thus $\tr B=0$ (see \eqref{eq:matrix}), $\|B\|_{L^2(A_5)}\le C(n)$, and the estimate $|th'(t)|\le C/L$ give
\begin{equation}\label{eq:trace-small}
 S:=\tr(DH(V)D\mathcal P(p)B)=w^{-1}h'(w)B[p,p],\qquad
 \|S\|_{L^2(A_3)}\le C/L.
\end{equation}
Note that \eqref{eq:H-composition} and \eqref{eq:trace-small} also hold on $\{w=0\}$.

For any $\xi,\Tilde{\xi}\in\RR^n$, we have the following elementary inequality
\[
 |\xi-\widetilde\xi|^2
 \le2|\mathcal P(\xi)-\mathcal P(\widetilde\xi)|.
\]
Set
$F_0(\xi)=h(|\xi|)\xi$. Since
\[
 DF_0(\xi)=h(|\xi|)I+|\xi|h'(|\xi|)\frac{\xi}{|\xi|}\otimes\frac{\xi}{|\xi|}
 \qquad(\xi\ne0),
\]
and $|th'(t)|\le C/L$, the map $F_0$ is globally Lipschitz with a constant independent of $b_0$. Moreover,
since $\mathcal P$ is onto, we can deduce that
\[
 \begin{aligned}
 |H(Z)-H(\widetilde Z)|^2
 &\le C|\xi-\widetilde\xi|^2
 \le C|\mathcal P(\xi)-\mathcal P(\widetilde\xi)|
 =C|Z-\widetilde Z|,
 \end{aligned}
\]
where $Z=\mathcal P(\xi)$ and $\widetilde Z=\mathcal P(\widetilde\xi)$. With $v_\sigma=H(V_\sigma)$, we can apply the above estimate with $Z=V_\sigma$, $\widetilde Z=V$, and the estimates \eqref{eq:V-mollified} and \eqref{eq:H-composition} to derive that
\begin{equation}\label{eq:inverse-approximation}
 \|p-v_\sigma\|_{L^2(A_2)}\le C\bigl(\tau+\Omega(\sigma,q)^{1/2}\bigr).
\end{equation}

Put
\[
 A:=D\mathcal P(p)B,
 \qquad
 A_\sigma(x)=\int_{B_\sigma}\varphi_\sigma(a)A(x-a)\,da.
\]
Convolving \eqref{eq:V-equation} gives
\[
 DV_\sigma=A_\sigma+(\Div E)_\sigma
 \qquad\hbox{on }A_2.
\]
Since $v_\sigma=H(V_\sigma)$ is $C^1$, direct computation yields
\[
 \begin{aligned}
 \operatorname{div}v_\sigma(x)
 &=\tr\bigl(DH(V_\sigma(x))A_\sigma(x)\bigr)
   +\tr\bigl(DH(V_\sigma(x))(\Div E)_\sigma(x)\bigr)\\
 &=\int_{B_\sigma}\varphi_\sigma(a)
   \tr\bigl(DH(V(x-a))A(x-a)\bigr)\,da\\
 &\quad+\int_{B_\sigma}\varphi_\sigma(a)
   \tr\Bigl((DH(V_\sigma(x))-DH(V(x-a)))A(x-a)\Bigr)\,da\\
 &\quad+\tr\bigl(DH(V_\sigma(x))(\Div E)_\sigma(x)\bigr).
 \end{aligned}
\]
Write
\begin{equation}\label{eq:approx-divergence}
 \operatorname{div}v_\sigma=S_\sigma+R_\sigma^B+R_\sigma^E
 =S_\sigma+R_\sigma,
\end{equation}
where, for $x\in A_2$,
\[
 \begin{aligned}
 S_\sigma(x)
 &:=\int_{B_\sigma}\varphi_\sigma(a)S(x-a)\,da,\\
 R_\sigma^B(x)
 &:=\int_{B_\sigma}\varphi_\sigma(a)
 \tr\Bigl((DH(V_\sigma(x))-DH(V(x-a)))
 D\mathcal P(p(x-a))B(x-a)\Bigr)\,da,\\
 R_\sigma^E(x)
 &:=\tr\bigl(DH(V_\sigma(x))(\Div E)_\sigma(x)\bigr),
 \qquad R_\sigma:=R_\sigma^B+R_\sigma^E.
 \end{aligned}
\]
For the rest of the proof we set
\begin{equation}\label{eq:CH-def}
 C_H:=C(n)\left(1+\|DH\|_{L^\infty(\RR^n)}^2
 +\|DH\|_{L^\infty(\RR^n)}\|D^2H\|_{L^\infty(\RR^n)}\right).
\end{equation}
By \eqref{eq:H-derivatives} and $0<b_0<1$, we know $ C_H\le C(n)b_0^{-4}$.

Note that componentwise,
\[
 [ (\Div E)_\sigma]_{ij}(x)
 =\sum_{k=1}^n\int_{B_\sigma}E_{ijk}(x-a)\,\partial_k\varphi_\sigma(a)\,da;
\]
thus from \eqref{eq:V-equation} we get
\begin{equation}\label{eq:DivE-mollified}
 \|(\Div E)_\sigma\|_{L^1(A_2)}
 \le \|E\|_{L^1(U)}\|D\varphi_\sigma\|_{L^1(\RR^n)}
 \le C\frac{\sqrt q}{\sigma}.
\end{equation}
From \eqref{eq:DivE-mollified} and the definition of $R_\sigma^E$, we obtain
\begin{equation}\label{eq:RE-bound}
 \|R_\sigma^E\|_{L^1(A_2)}
 \le C\|DH\|_{L^\infty}\|(\Div E)_\sigma\|_{L^1(A_2)}
 \le C_H\frac{\sqrt q}{\sigma}.
\end{equation}

As for the term $R_\sigma^B$, we consider the measure
$d\mu=dx\,\varphi_\sigma(a)\,da$ on $A_2\times B_\sigma$ and write $y=x-a$. Set
\[
 G(x,a):=|DH(V_\sigma(x))-DH(V(y))|^2.
\]
Clearly $G\leq C_H$. For any $Z,\widetilde Z\in\RR^n$, we have
\[
 |DH(Z)-DH(\widetilde Z)|^2
 \le C_H|Z-\widetilde Z|,
\]
which, together with \eqref{eq:V-double}, implies
\begin{equation}\label{eq:G-bounds}
 \int G\,d\mu\le C_H\Omega(2\sigma,q).
\end{equation}
Since $w^2\le k+e_0$, H\"older's inequality, \eqref{eq:energy-majorant} and \eqref{eq:G-bounds} give
\[
 \begin{aligned}
 \int Gw(y)^2\,d\mu
 &\le \int Gk(y)\,d\mu+\int Ge_0(y)\,d\mu\\
 &\le C\|G\|_{L^n(d\mu)}+Cq^{1/4}\|G\|_{L^\infty(d\mu)}\\
 &\le C_H\bigl(\Omega(2\sigma,q)^{1/n}+q^{1/4}\bigr).
 \end{aligned}
\]
Finally, we can combine this estimate with H\"older's inequality to get
\[
 \begin{aligned}
 \|R_\sigma^B\|_{L^1(A_2)}
 &\le C\int_{A_2\times B_\sigma}G^{1/2}w(y)|B(y)|\,d\mu\\
 &\le C\left(\int Gw(y)^2\,d\mu\right)^{1/2}
       \left(\int|B(y)|^2\,d\mu\right)^{1/2}\\
 &\le C_H\bigl(\Omega(2\sigma,q)^{1/(2n)}+q^{1/8}\bigr),
 \end{aligned}
\]
which, together with \eqref{eq:RE-bound}, gives
\begin{equation}\label{eq:commutator}
 \|R_\sigma\|_{L^1(A_2)}
 \le C_H\left(\Omega(2\sigma,q)^{1/(2n)}+q^{1/8}+\frac{\sqrt q}{\sigma}\right).
\end{equation}

Let $c=u_{A_3}$. From Poincar\'e's inequality we see $\|u-c\|_{L^2(A_3)}\le C$. Fix a radial function $\theta\in C_c^\infty(A_2)$, equal to one on $A_1$, with $0\le\theta\le1$. Choose a smooth even function $\zeta$ which is one on $[-1,1]$, zero outside $(-2,2)$, and satisfies $0\leq \zeta\leq1$. For $M\ge1$, set
\[
 T_M(s)=\int_0^s\zeta(t/M)\,dt,\qquad g=\theta T_M(u-c).
\]
Since $0\le\zeta\le1$, we have
\[
 |T_M(s)|\le\min\{|s|,2M\},\qquad |T_M'(s)|=|\zeta(s/M)|\le1.
\]
Since
\begin{align}\label{add5}
    Dg=T_M(u-c)D\theta+\theta T_M'(u-c)p,
\end{align}
it follows that
\begin{equation}\label{eq:bounded-test}
 \|g\|_{L^2(A_2)}+\|Dg\|_{L^2(A_2)}\le C,
 \qquad \|g\|_{L^\infty(A_2)}\le2M.
\end{equation}

Consider the set
\[
 E_M:=A_2\cap\{|u-c|>M\}.
\]
By Chebyshev's inequality and the estimate $\|u-c\|_{L^2(A_3)}\le C$, we have
\[
 |E_M|\le M^{-2}\|u-c\|_{L^2(A_2)}^2\le CM^{-2}.
\]
Applying \eqref{eq:small-set-energy} with $E=E_M$ then yields
\begin{align}
    \int_{E_M}w^2
 \le C|E_M|^{1/n}+Cq^{1/4}
 \le CM^{-2/n}+Cq^{1/4}.\nonumber
\end{align}
As $T'_M=1$ on $[-M,M]$ and $T'_M\ge0$,
\begin{align}\label{add6}
    \int_{A_1}w^2
 \le\int_{A_2}\theta T'_M(u-c)w^2+CM^{-2/n}+Cq^{1/4}.
\end{align}
From \eqref{add5}, the first integral on the right can be written as
\begin{align}\label{add5.5}
    \int_{A_2}\theta T'_M(u-c)w^2
 =\int_{A_2}p\cdot Dg
   -\int_{A_2}T_M(u-c)p\cdot D\theta.
\end{align}
As $\theta$ is radial, $p\cdot D\theta=\theta'(|x|)z/|x|$. Thus we have by H\"older's inequality
\begin{align}\label{add7}
 \left|\int_{A_2}T_M(u-c)p\cdot D\theta\right|
 &\le C\int_{A_2}|u-c||z|\nonumber\\
 &\le C\|u-c\|_{L^2(A_2)}\|z\|_{L^2(A_2)}\le Cd.
\end{align}

To estimate the integral of $p\cdot Dg$, we replace $p$ by $v_\sigma$. Specifically, by \eqref{eq:inverse-approximation} and \eqref{eq:bounded-test}, we have
 \begin{align}\label{add8}
 \left|\int_{A_2}(p-v_\sigma)\cdot Dg\right|
 &\le \|p-v_\sigma\|_{L^2(A_2)}\|Dg\|_{L^2(A_2)}\nonumber\\
 &\le C\bigl(\tau+\Omega(\sigma,q)^{1/2}\bigr).
 \end{align}
Since $g$ is compactly supported in $A_2$, integration by parts and \eqref{eq:approx-divergence} give
\[
 \begin{aligned}
 \int_{A_2}v_\sigma\cdot Dg
 &=-\int_{A_2}g\,\operatorname{div}v_\sigma\\
 &=-\int_{A_2}gS_\sigma-\int_{A_2}gR_\sigma.
 \end{aligned}
\]
From \eqref{eq:trace-small} 
and \eqref{eq:bounded-test} we can obtain
\begin{align}\label{add9}
    \left|\int_{A_2}gS_\sigma\right|
    \le \|g\|_{L^{2}(A_2)}\|S_\sigma\|_{L^2(A_2)}
    \le C\|S\|_{L^2(A_3)}
    \le C/L.
\end{align}
Besides, \eqref{eq:bounded-test} and \eqref{eq:commutator} give
\begin{align}\label{add10}
    \left|\int_{A_2}gR_\sigma\right|
 \le2MC_H\left(\Omega(2\sigma,q)^{1/(2n)}+q^{1/8}+\frac{\sqrt q}{\sigma}\right).
\end{align}

Combining the estimates \eqref{add6}-\eqref{add10}, and recalling
$\Omega(s,q)=C(s^{1/4}+q^{1/4})$, we obtain
\[
 \begin{aligned}
 \int_{A_1}w^2
 \le C(n)\bigl(M^{-2/n}+\tau+L^{-1}+d\bigr)
 +MC_H\left(\sigma^{1/(8n)}+q^{1/(8n)}+\frac{\sqrt q}{\sigma}\right).
 \end{aligned}
\]
To complete the proof, we just need to first take $M,L$ large, $\tau,d$ small, then take $\sigma$ small, and finally take $q$ small. 
\end{proof}

We next establish an energy decay estimate for stable solutions in every dimension $n\ge2$, which plays a key role in the proofs of Theorems~\ref{thm:gradient} and~\ref{thm:holder}. The cases $2\le n\le5$ and $n\ge10$ follow from Lemmas~\ref{lem:rotation} and~\ref{lem:radial}, respectively.

For the remaining dimensions $6\le n\le9$, we consider the rescaled functions $v_a(x):=u(ax)$. Applying Lemma~\ref{lem:radial}, we obtain a uniform bound for the integral of $\|x\cdot Dv_a\|_{L^2(A_6)}^2+Q_{v_a}((x\cdot Dv_a)\eta)$ with respect to $da/a$; see \eqref{eq:scale-average}. Averaging over a suitable range of scales then allows us to choose $a$ such that, after normalization, the hypotheses of Proposition~\ref{prop:rigidity} are satisfied. This yields a fixed-scale contraction of the gradient energy, which can be iterated to obtain the desired decay estimate.

\begin{proposition}\label{prop:decay}
Let $u\in C^{2}(B_R(y))$ be a stable solution of \eqref{eq:pde} in $B_R(y)$. With $\beta_n$ defined by \eqref{eq:beta}, it holds that
\begin{equation}\label{eq:energy-decay}
 \int_{B_r(y)}|Du|^2\le C(n)(r/R)^{n-\beta_n}
                    \int_{B_R(y)}|Du|^2\qquad(0<r\le R/2).
\end{equation}
Consequently, for $0<r<R$,
\begin{equation}\label{eq:critical-morrey}
 \|Du\|_{M^{2,\beta_n}(B_r(y))}\le C(n,r,R)\|Du\|_{L^2(B_R(y))}.
\end{equation}
\end{proposition}
\begin{proof}
It remains to deal with the case $6\le n\le9$. By translation and scaling, it suffices to work in $B_1$. Set
\[
 F_0(t)=\int_{B_t}|Du|^2,\qquad
 D(t)=t^{2-n}\int_{B_t\setminus B_{t/3}}|Du|^2,\qquad
 \mathcal R(t)=\int_{B_t}|x|^{2-n}u_r^2.
\]
Lemma~\ref{lem:radial} gives constants $K=K(n)\ge1$ and $b=b(n)>0$ such that
\begin{equation}\label{eq:cross-scale}
 \begin{gathered}
 s^{2-n}F_0(s)\le KD(t),\qquad D(s)\le KD(t)\quad(0<s\le t/2),\\
 \mathcal R(t/3)\le bD(t).
 \end{gathered}
\end{equation}
We may assume $\int_{B_1}|Du|^2<\infty$, since otherwise \eqref{eq:energy-decay} is trivial. If $D(1)=0$, then \eqref{eq:cross-scale} gives $F_0(1/2)=0$, so \eqref{eq:energy-decay} follows immediately. Otherwise, replacing $u$ by $u/\sqrt{D(1)}$ and $f$ by $t\mapsto f(\sqrt{D(1)}\,t)/\sqrt{D(1)}$, we may assume $D(1)=1$.

Fix the radial cutoff function $\eta$ in Proposition~\ref{prop:rigidity} and put $a_+=1/256$. For $0<a\le a_+$, we set $v_a(x)=u(ax)$, $z_a=x\cdot Dv_a$, and $q(a)=Q_{v_a}(z_a\eta)$. By direct computation we have
\begin{equation}\label{eq:jacobi-scaling}
 q(a)
 =a^{2-n}Q_u\bigl(x\cdot Du(x)\eta(x/a)\bigr)\geq 0.
\end{equation}
We claim that there exists a constant $J=J(n)<\infty$ such that
\begin{equation}\label{eq:scale-average}
 \int_{a_-}^{a_+}\left(\|z_a\|_{L^2(A_6)}^2+q(a)\right)\frac{da}{a}\le J
 \qquad(\forall\;0<a_-<a_+).
\end{equation}
To deduce it, first note that \eqref{eq:cross-scale} gives $\mathcal R(1/4)\le b$. The change of variables $y=ax$ gives
\[
 \begin{aligned}
 \int_{a_-}^{a_+}\|z_a\|_{L^2(A_6)}^2\frac{da}{a}
 &\le\int_{B_{1/4}}r^2u_r^2
              \left(\int_{r/64}^{64r}a^{-n}\frac{da}{a}\right)dy\\
 &=\frac{64^n-64^{-n}}n\mathcal R(1/4)\le C(n).
 \end{aligned}
\]
To estimate $q(a)$, from \eqref{eq:jacobi-scaling} and \eqref{eq:dilation} we can compute that 
\[
 \begin{aligned}
 q(a)={}&\int a^{-n}r^2u_r^2|\eta'(r/a)|^2\,dx\\
 &+\int a^{2-n}|Du|^2
 \left((2-n)\eta(r/a)^2-\frac ra(\eta^2)'(r/a)\right)dx\\
 &+2\int a^{1-n}ru_r^2(\eta^2)'(r/a)\,dx.
 \end{aligned}
\]

For the first term, we have
\[
 \begin{aligned}
 &\int_{a_-}^{a_+}\int a^{-n}r^2u_r^2|\eta'(r/a)|^2\,dx\,\frac{da}{a}\\
 &\qquad\le\int_{B_{1/4}}r^2u_r^2
 \left(\int_0^\infty a^{-n}|\eta'(r/a)|^2\frac{da}{a}\right)dx\\
 &\qquad=\left(\int_0^\infty t^{n-1}|\eta'(t)|^2dt\right)
 \int_{B_{1/4}}r^{2-n}u_r^2\,dx
 \le C(n).
 \end{aligned}
\]

For the middle term, we set
\[
 \Psi(r)=\int_{a_-}^{a_+}a^{2-n}\eta(r/a)^2\frac{da}{a},\qquad
 H(t)=t^{n-2}\eta(t)^2.
\]
Then
\[
 r\Psi'(r)=\int_{a_-}^{a_+}a^{2-n}\frac ra(\eta^2)'(r/a)\frac{da}{a},
\]
and integration in $a$ transforms the middle term into
\[
 \int |Du|^2\bigl((2-n)\Psi(r)-r\Psi'(r)\bigr)\,dx.
\]
As we have
\[
 (2-n)\Psi(r)-r\Psi'(r)
 =r^{2-n}\bigl[H(r/a_+)-H(r/a_-)\bigr],
\]
and $H(r/a_+)\ne0$ only for
\[
 \ell:=a_+/64<r<h:=64a_+,
\]
the middle term is bounded above by
\[
 \|H\|_{L^\infty}\int_{\ell<|x|<h}|x|^{2-n}|Du|^2\,dx.
\]
Thanks to \eqref{eq:cross-scale}, we see it is further bounded by $C(n)$. 

For the third term, its absolute value is bounded by
\[
 \begin{aligned}
 &2\int_{a_-}^{a_+}\int a^{1-n}r u_r^2
 |(\eta^2)'(r/a)|\,dx\,\frac{da}{a}\\
 &\qquad\le
 2\left(\int_0^\infty t^{n-2}|(\eta^2)'(t)|dt\right)
 \int_{B_{1/4}}r^{2-n}u_r^2\,dx\\
 &\qquad\le C(n)\mathcal R(1/4)\leq C(n).
 \end{aligned}
\]
Combining the above estimates then yields \eqref{eq:scale-average}.

Next we set
\[
 M_0=64^{n-2}K,\qquad \mu_0=\frac{2^{n-3}}K,\qquad
 \varepsilon_* =\frac{\mu_0}{2M_0}.
\]
Take $\delta_*$ from Proposition~\ref{prop:rigidity} for this fixed $\varepsilon_*$, and define
\[
 T=1+\frac{2J}{M_0\delta_*^2},\qquad
 \vartheta=2^{-14}e^{-T},\qquad a_-=64\vartheta.
\]
Then $\log(a_+/a_-)=T$. Suppose that $D(\vartheta)>1/2$. Then for every $a\in[a_-,a_+]$, \eqref{eq:cross-scale} gives
\[
 \int_{A_6}|Dv_a|^2\le a^{2-n}F_0(64a)\le M_0,
 \qquad
 \int_{A_1}|Dv_a|^2\ge2^{n-2}D(2a)>\mu_0.
\]
By \eqref{eq:scale-average}, there exists $a\in [a_-,a_+]$ such that
\[
 \|z_a\|_{L^2(A_6)}^2+q(a)\le J/T<M_0\delta_*^2/2.
\]
Set $w:=v_a/\sqrt{M_0}$; then we can apply
Proposition~\ref{prop:rigidity} to deduce that $\int_{A_1}|Dw|^2\le \varepsilon_*$, which is equivalent to $\int_{A_1}|Dv_a|^2\le\mu_0/{2}$. This yields a contradiction. Therefore, we have $D(\vartheta)\leq1/2$. By rescaling we know that for every $0<t\leq 1$,
\begin{equation}\label{eq:contraction}
 D(\vartheta t)\le\tfrac12D(t).
\end{equation}

Now we choose
\begin{align}\label{add11}
 \alpha_n=\frac{\log2}{2\log(1/\vartheta)}.
\end{align}
Iteration of \eqref{eq:contraction} gives $D(\vartheta^k)\le2^{-k}D(1)$. Given $0<s\le1/2$, choose $k\ge0$ with $\vartheta^{k+1}/2<s\le\vartheta^k/2$. By \eqref{eq:cross-scale} we have
\[
 s^{2-n}F_0(s)\le K2^{-k}D(1)
 \le K2^{1+2\alpha_n}s^{2\alpha_n}D(1).
\]
Since $D(1)\le F_0(1)$, we arrive at \eqref{eq:energy-decay}. A standard scaling and covering argument then yields \eqref{eq:critical-morrey}. The proof is complete.
\end{proof}

\section{Interior regularity and sharpness}\label{sec3}

This section is devoted to the proofs of Theorems \ref{thm:gradient}, \ref{thm:holder}, and \ref{thm:obstructions}, based on the tools established in Section \ref{sec2}.

\subsection{Gradient and Hessian estimates}

\begin{proof}[Proof of Theorem~\ref{thm:gradient}]
We first prove that
\begin{equation}\label{add12}
 \|Du\|_{L^{2+\gamma_{n,q}}(B_{1/2})}
 +\|D^2u\|_{L^1(B_{1/2})}
 \le C(n,q,\gamma_{n,q})A_0,
\end{equation}
where
\[
 A_0:=\|Du\|_{L^2(B_{7/8})}+\|f^-(u)\|_{L^q(B_{7/8})}.
\]
Note that the Hessian part follows from Lemma~\ref{lem:endpoint}. In the following, we focus on the gradient part.

Put $w=|Du|$ and fix $\sigma>1$ and $1/2\le r<t\le2/3$. Choose
$\chi\in C_c^\infty(B_t)$ with $0\le\chi\le1$, $\chi\equiv1$ on $B_r$, and
$|D\chi|\le C(t-r)^{-1}$. On $\{w>0\}$ we write $e=Du/w$ and set
$T_0=\Delta u-D^2u[e,e]$. Note that $|T_0|\le C(n)\mathcal A_u$, and the identity
\[
 \operatorname{div}(w^{\sigma-1}Du)
 =\sigma w^{\sigma-1}\Delta u-(\sigma-1)T_0w^{\sigma-1}
\]
holds weakly, with zero values on $\{w=0\}$. By \eqref{eq:matrix} and \eqref{eq:laplace-sign}, we have
\[
 \begin{aligned}
 \int_{B_r}w^{\sigma-1}|D^2u|
 &\le \int_{B_t}\chi w^{\sigma-1}|\Delta u|
      +C(n)\int_{B_t}\chi\mathcal A_uw^{\sigma-1}\\
 &=-\int_{B_t}\chi w^{\sigma-1}\Delta u
      +2\int_{B_t}\chi f^-(u)w^{\sigma-1}
      +C(n)\int_{B_t}\chi\mathcal A_uw^{\sigma-1}.
 \end{aligned}
\]
The divergence identity gives
\[
 \begin{aligned}
 -\int_{B_t}\chi w^{\sigma-1}\Delta u
 &= -\frac1\sigma\int_{B_t}\chi\,\operatorname{div}(w^{\sigma-1}Du)
    -\frac{\sigma-1}{\sigma}\int_{B_t}\chi T_0w^{\sigma-1}\\
 &= \frac1\sigma\int_{B_t}w^{\sigma-1}Du\cdot D\chi
    -\frac{\sigma-1}{\sigma}\int_{B_t}\chi T_0w^{\sigma-1}\\
 &\le \frac{C}{t-r}\int_{B_t}w^\sigma
    +C(n,\sigma)\int_{B_t}\mathcal A_uw^{\sigma-1}.
 \end{aligned}
\]
Consequently, we get
\begin{equation}\label{eq:weighted-hessian}
 \begin{aligned}
 \int_{B_r}w^{\sigma-1}|D^2u|
 \le C(n,\sigma)\biggl(&\int_{B_t}f^-(u)w^{\sigma-1}
       +\int_{B_t}\mathcal A_u w^{\sigma-1}\\
       &+(t-r)^{-1}\int_{B_t}w^\sigma\biggr).
 \end{aligned}
\end{equation}

When $n=2$, it suffices to consider the case $1\leq q<2$. By \eqref{eq:endpoint-estimate}, we have $\|Du\|_{L^4(B_{3/4})}\le C A_0$. This proves the desired estimate for $q\le4/3$. For $4/3<q<2$, set $\sigma=q/(2-q)>2$; then $2+\gamma_{2,q}=2\sigma$. Define $s=(r+t)/2$ and choose
$\eta\in C_c^\infty(B_{s})$ with $\eta\equiv1$ on $B_r$ and
$|D\eta|\le C(t-r)^{-1}$. By applying the $W^{1,1}$ Sobolev inequality to $V:=\eta w^{\sigma-1}Du$, we derive
\[
 \begin{aligned}
 X(r)^\sigma
 &=\left(\int_{B_r}w^{2\sigma}\right)^{1/2}
 \le \|V\|_{L^2(\RR^2)}
 \le C\|DV\|_{L^1(\RR^2)},
 \end{aligned}
\]
where $X(\rho):=\|Du\|_{L^{2\sigma}(B_\rho)}$. Moreover,
\[
 |DV|\le C(\sigma)\eta w^{\sigma-1}|D^2u|+|D\eta|w^\sigma,
\]
because $|D(w^{\sigma-1}Du)|\le C(\sigma)w^{\sigma-1}|D^2u|$. Hence
\[
 X(r)^\sigma
 \le C\int_{B_s}w^{\sigma-1}|D^2u|
      +C(t-r)^{-1}\int_{B_s}w^\sigma.
\]
Applying \eqref{eq:weighted-hessian} with $r$ replaced by $s$ then gives
\begin{equation}\label{eq:two-dimensional-bootstrap}
 X(r)^\sigma
 \le C(t-r)^{-1}\left(I_1+I_2+I_3\right),
\end{equation}
with
\[
 I_1:=\int_{B_t}f^-(u)w^{\sigma-1},\qquad
 I_2:=\int_{B_t}\mathcal A_uw^{\sigma-1},\qquad
 I_3:=\int_{B_t}w^\sigma.
\]
For $I_1$, the identity $(\sigma-1)q'=2\sigma$ yields
\[
 I_1\le \|f^-(u)\|_{L^q(B_t)}X(t)^{\sigma-1}
 \le A_0X(t)^{\sigma-1}.
\]
For $I_2$, from \eqref{eq:geometric} and H\"older's inequality we get
\[
 I_2\le \|\mathcal{A}_u\|_{L^2(B_t)}\|w^{\sigma-1}\|_{L^2(B_t)}\leq CA_0X(t)^{\sigma-1}.
\]
For $I_3$, if $\sigma\le4$, from \eqref{eq:endpoint-estimate} and H\"older's inequality we get $I_3\le CA_0^\sigma$. If $\sigma>4$, by interpolating between $L^4$ and $L^{2\sigma}$:
\[
 \|w\|_{L^\sigma(B_t)}
 \le C\|w\|_{L^4(B_t)}^\theta\|w\|_{L^{2\sigma}(B_t)}^{1-\theta},
 \qquad \theta=\frac2{\sigma-2},
\]
and applying \eqref{eq:endpoint-estimate}, we can deduce
\[
 I_3\le CA_0^aX(t)^{\sigma-a},
 \qquad a=\sigma\theta=\frac{2\sigma}{\sigma-2}\in(0,\sigma).
\]
Substituting the above estimates in \eqref{eq:two-dimensional-bootstrap} and applying Young's inequality, we obtain, for every $\varepsilon\in(0,1)$,
\begin{equation}\label{eq:two-dimensional-small-constant}
 X(r)^\sigma\le\varepsilon X(t)^\sigma
       +C(\varepsilon)(t-r)^{-b'}A_0^\sigma,
 \qquad \frac12\le r<t\le\frac23,
\end{equation}
for some finite constant $b'=b'(q)$. By choosing $\varepsilon$ small, we can apply Simon's lemma, as in the proof of Lemma~\ref{lem:endpoint}, to obtain that
$X(1/2)\le CA_0$. 

For $n\ge3$, we put $\beta=\beta_n$ for simplicity. Recall that Proposition~\ref{prop:decay} gives
\[
 \|Du\|_{M^{2,\beta}(B_{3/4})}\le C A_0.
\]
By applying Lemma~\ref{lem:morrey-bv} with $p=2$ to a cutoff of $Du$ over $B_\frac{3}{4}$ and using \eqref{eq:endpoint-estimate}, we obtain
\[
 \|Du\|_{L^{1+2/\beta}(B_{2/3})}^{1+2/\beta}
 \le C A_0^{2/\beta}
       \left(\|D^2u\|_{L^1(B_{3/4})}+\|Du\|_{L^1(B_{3/4})}\right)
 \le C A_0^{1+2/\beta}.
\]
Together with the $L^{2n/(n-1)}$ bound in Lemma~\ref{lem:endpoint}, we know that it remains to consider the case
$\gamma_{n,q}+2>\max\{2n/(n-1),1+2/\beta\}$.

Set $m=2+\gamma_{n,q}$. Then $q>1$ and $\sigma:=m-\frac2\beta>1$. Put $s=(r+t)/2$, choose
$\eta\in C_c^\infty(B_s)$ with $\eta\equiv1$ on $B_r$ and
$|D\eta|\le C(t-r)^{-1}$, and set $V:=\eta w^{\sigma-1}Du$. Since $|V|=\eta w^\sigma$, we have
\[
 \|V\|_{M^{2/\sigma,\beta}(\RR^n)}^{2/\sigma}
 \le C\|Du\|_{M^{2,\beta}(B_{3/4})}^2\le CA_0^2.
\]
Put $X(\rho):=\|Du\|_{L^m(B_\rho)}$. Applying Lemma~\ref{lem:morrey-bv} with $p=2/\sigma$ then yields
\[
 X(r)^m\le CA_0^{2/\beta}\int_{\RR^n}|DV|.
\]
As in the two-dimensional case, we have
\[
 |DV|\le C(\sigma)\eta w^{\sigma-1}|D^2u|+|D\eta|w^\sigma,
\]
and thus
\[
 X(r)^m
 \le CA_0^{2/\beta}\left(
 \int_{B_s}w^{\sigma-1}|D^2u|
 +(t-r)^{-1}\int_{B_s}w^\sigma\right).
\]
Applying \eqref{eq:weighted-hessian} with $r$ replaced by $s$ then gives
\begin{equation}\label{eq:gradient-bootstrap}
 \begin{aligned}
 X(r)^m\le C(t-r)^{-1}A_0^{2/\beta}\biggl(&\int_{B_t}f^-(u)w^{\sigma-1}
           +\int_{B_t}\mathcal A_u w^{\sigma-1}
           +\int_{B_t}w^\sigma\biggr).
 \end{aligned}
\end{equation}
Moreover, as the choices in \eqref{eq:gamma} show that
\begin{equation}\label{eq:exponent-check}
 (\sigma-1)q'\le m,\qquad 2\sigma-2\le m,\qquad \sigma<m,
\end{equation}
we can argue as in the case $n=2$ to deduce that, for every $\varepsilon\in(0,1)$,
\begin{equation}\label{eq:high-dimensional-small-constant}
 X(r)^m\le\varepsilon X(t)^m
       +C(\varepsilon)(t-r)^{-b'}A_0^m
\end{equation}
for some finite constant $b'=b'(n,q)$. Then by choosing $\varepsilon$ sufficiently small in terms of $b'$, we can apply Simon's lemma, as in the proof of Lemma~\ref{lem:endpoint}, to obtain that $X(1/2)\le CA_0$.

Now we have established \eqref{add12}. To pass from \eqref{add12} to \eqref{eq:main-gradient}, we first apply a standard scaling and covering argument to deduce that
\[
 \|Du\|_{L^{2+\gamma_{n,q}}(B_{1/2})}+\|D^2u\|_{L^1(B_{1/2})}
 \le C\left(\|Du\|_{L^2(B_{5/8})}+\|f^-(u)\|_{L^q(B_{5/8})}\right).
\]
Then by \eqref{eq:annulus-data} with $r=5/8$, $s=3/4$, and $c=0$, we see that the term $\|Du\|_{L^2(B_{5/8})}$ is bounded by $C(n)\mathfrak D(u)$. This gives \eqref{eq:main-gradient}. \eqref{eq:main-endpoint} then follows by taking $q=1$.
\end{proof}

\subsection{H\"older regularity}

\begin{proof}[Proof of Theorem~\ref{thm:holder}]
Fix $0<r<R<1$. For $x\in B_r$ and $0<\rho\le (R-r)/2$, applying \eqref{eq:energy-decay} on $B_{R-r}(x)\subseteq B_R$ and using $n-\beta_n=n-2+2\alpha_n$ gives
\[
 \int_{B_\rho(x)}|Du|^2
 \le C(n,r,R)\rho^{n-2+2\alpha_n}\|Du\|_{L^2(B_R)}^2.
\]
By H\"older's inequality we get
\[
 \int_{B_\rho(x)\cap B_r}|Du|
 \le C(n,r,R)\rho^{n-1+\alpha_n}\|Du\|_{L^2(B_R)}.
\]
For $\rho>(R-r)/2$, the latter estimate follows from $\int_{B_\rho(x)\cap B_r}|Du|\le |B_r|^{1/2}\|Du\|_{L^2(B_R)}$, after enlarging $C(n,r,R)$. The Morrey estimate of Gilbarg--Trudinger \cite[Thm.\,7.19]{GilbargTrudinger2001} therefore yields
\begin{equation}\label{eq:holder-energy}
 [u]_{C^{0,\alpha_n}(\ol{B_r})}
 \le C(n,r,R)\|Du\|_{L^2(B_R)}.
\end{equation}
Taking $r=1/2$ and $R=3/4$ proves \eqref{eq:main-holder0}.

Applying Proposition~\ref{prop:annulus}, with $r=15/16$, $s=3/4$, and $c=0$, and using \eqref{eq:holder-energy} with $r=7/8$ and $R=15/16$ yield
\[
 [u]_{C^{0,\alpha_n}(\ol{B_{7/8}})}\le C(n)\mathfrak D(u).
\]
For $x\in B_{1/2}$, we can average $|u(x)|\le |u(z)|+2^{\alpha_n}[u]_{C^{0,\alpha_n}(\ol{B_{7/8}})}$ over $z\in B_{7/8}\setminus\ol{B_{3/4}}$ to get
\[
 \|u\|_{L^\infty(B_{1/2})}\le C(n)\mathfrak D(u).
\]
Combining the above two estimates gives \eqref{eq:main-holder}.
\end{proof}

\subsection{Obstructions and sharpness}

\begin{proof}[Proof of Theorem~\ref{thm:obstructions}]
For part~(i), set
\[
 u_j(x)=\tanh(jx_1),\qquad f_j(t)=2j^2t(1-t^2).
\]
Clearly $|u_j|\le1$. Besides, we can compute that $-\Delta u_j=f_j(u_j)$, and
$\phi_j:=\partial_1u_j=j\operatorname{sech}^2(jx_1)>0$ satisfies
$-\Delta\phi_j=f_j'(u_j)\phi_j$. From  Lemma~\ref{lem:ground-state} we know that $u_j$ is stable. For $1<p<\infty$, we have
\[
 \int_{B_{1/2}}|Du_j|^p\ge c(n,p)j^{p-1}
                         \int_{-1}^1\operatorname{sech}^{2p}t\,dt.
\]
Also $\|Du_j\|_{L^\infty(B_{1/2})}=j$. Since
\[
 |\partial_{11}u_j|=2j^2|\tanh(jx_1)|\operatorname{sech}^2(jx_1)
\]
on $1/(2j)<x_1<1/j$, we can obtain $\|D^2u_j\|_{L^1(B_{1/2})}\ge c(n)j$. For large $j$, evaluating the value of $u_j$ at two points $\pm e_1/j$ gives
\[
 [u_j]_{C^{0,\alpha}(\ol{B_{1/2}})}
 \ge 2^{1-\alpha}\tanh(1)j^\alpha.
\]
Therefore, $u_j,f_j$ satisfy our requirements.

\vskip0.1in

For part~(ii), choose $\psi\in C_c^\infty((-1,1))$ with $0\le\psi\le1$ and $\psi=1$ on $[-1/2,1/2]$. Let $0<c\le1/16$ be a dimensional constant to be fixed below, and set
\[
 b_j(r)=1+j\psi\left(\frac{j^2(r-1/4)}c\right),\qquad
 u_j(x)=\int_{|x|}^1t b_j(t)\,dt.
\]
By definition we can check that $u_j$ is smooth, radial, $0\leq u_j\leq1$, $u_j=0$ on $\partial B_1$, and $u_j(x)=(1-|x|^2)/2$ for $|x|\ge3/8$. 

As $u_j$ is strictly decreasing radially, we can define $f_j$ on the range of $u_j$ by
\begin{equation}\label{eq:radial-profile}
 f_j\left(\int_r^1t b_j(t)\,dt\right)=n b_j(r)+r b_j'(r).
\end{equation}
Since the right-hand side is identically $n$ when $r$ is close to $0$ or $1$, we can smoothly extend $f_j$ to a smooth function in $\RR$. Note that \eqref{eq:radial-profile} is exactly $-\Delta u_j=f_j(u_j)$. In particular, we have $f_j^-(u_j)=0$ on $B_1\setminus B_{3/4}$.

To show the stability of $u_j$, we take $\xi\in C_c^1(B_1\setminus\{0\})$ and write $\xi=r b_j(r)\zeta$. Then we have
\[
 (-\Delta-f_j'(u_j))(r b_j)=-(n-1)r^{-2}(r b_j),
\]
and
\[
 Q_{u_j}(\xi)=\int_{B_1} r^2b_j^2|D\zeta|^2-(n-1)\int_{B_1} b_j^2\zeta^2.
\]
Since $b_j\ge1$,
\[
 \begin{aligned}
 Q_{u_j}(\xi)
 &\ge \int_{B_1}r^2|D\zeta|^2-(n-1)\int_{B_1}\zeta^2
      -(n-1)\int_{B_1}(b_j^2-1)\zeta^2\\
 &=\int_{B_1}|D(r\zeta)|^2-(n-1)\int_{B_1}(b_j^2-1)\zeta^2.
 \end{aligned}
\]
For $3/16\le r\le5/16$, we can write
\[
 r\zeta(r\omega)=-\int_r^1\partial_t\bigl(t\zeta(t\omega)\bigr)\,dt,
\]
and by H\"older's inequality, we get
\[
 \begin{aligned}
 |\zeta(r\omega)|^2
 &\le r^{-2}\left(\int_r^1t^{1-n}\,dt\right)
       \int_r^1t^{n-1}|\partial_t(t\zeta(t\omega))|^2\,dt\\
 &\le C(n)\int_0^1t^{n-1}|\partial_t(t\zeta(t\omega))|^2\,dt.
 \end{aligned}
\]
Let $I_j=(1/4-c/j^2,1/4+c/j^2)$. Since $b_j^2-1\le3j^2$ on $I_j$ and vanishes outside it, it holds that
\[
 \begin{aligned}
 \int_{B_1}(b_j^2-1)\zeta^2
 &\le 3j^2\int_{\mathbb S^{n-1}}\int_{I_j}|\zeta(r\omega)|^2r^{n-1}\,dr\,d\omega\\
 &\le6cC(n)\int_{\mathbb S^{n-1}}\int_0^1 t^{n-1}|\partial_t(t\zeta(t\omega))|^2\,dt\,d\omega\\
 &\le6cC(n)\int_{B_1}|D(r\zeta)|^2.
 \end{aligned}
\]
By choosing $c$ such that $6cC(n)(n-1)\le1/2$, we derive $Q_{u_j}(\xi)\ge0$. A simple approximation argument then yields $Q_{u_j}\geq 0$ for every $C_c^1$ function, which means that $u_j$ is stable.

Since $u_j'(r)=-rb_j(r)$, $b_j=1$ outside $I_j$ and $b_j^2\le1+3j^2$ on $I_j$, we have
\[
 \begin{aligned}
 \int_{B_1}|Du_j|^2
 &=|\mathbb S^{n-1}|\int_0^1r^{n+1}b_j(r)^2\,dr\\
 &\le C(n)\left(1+j^2|I_j|\right)\le C(n,c).
 \end{aligned}
\]
On $|r-1/4|\le c/(2j^2)$ one has $b_j(r)=1+j$ and $r\ge3/16$, so $|Du_j|=r(1+j)\ge3j/16$. Since this interval has length $c/j^2$, we have
\[
 \int_{B_{1/2}}|Du_j|^p\ge c(n,p)j^{p-2}\qquad(p>2).
\]
Besides, with $a_j:=1/4-c/j^2$, we can compute that
\[
 \begin{aligned}
 \|D^2u_j\|_{L^1(B_{1/2})}
 &\ge c(n)\int_{a_j}^{1/4}|(rb_j(r))'|\,dr\\
 &\ge c(n)\left|\frac14b_j(1/4)-a_jb_j(a_j)\right|.
 \end{aligned}
\]
Since $b_j(1/4)=1+j$ and $b_j(a_j)=1$, the last quantity equals $c(n)(j/4+c/j^2)\ge c(n)j$. Therefore, $u_j,f_j$ satisfy our requirements.

\vskip0.1in

For part~(iii), take $\rho\in C_c^\infty((-1,1))$, $\rho\ge0$, with $\int_{-1}^1\rho(s)\,ds=1$, and define
\[
 q_j(s)=\tfrac14j\rho(js),\qquad
 a_j(s)=1-\int_{-\infty}^s q_j(t)\,dt,\qquad
 \varphi_j(s)=\int_0^s a_j(t)\,dt.
\]
Because $3/4\le a_j\le1$, the function $\varphi_j$ is a smooth increasing diffeomorphism of $\RR$. Set
\[
 u_j(x)=\varphi_j(x_1),\qquad f_j(t)=q_j(\varphi_j^{-1}(t)).
\]
Then $f_j$ is smooth, compactly supported, and nonnegative, and
$-\Delta u_j=f_j(u_j)$. Since the positive function $\partial_1u_j=a_j(x_1)$ satisfies $-\Delta a_j=f_j'(u_j)a_j$, Lemma~\ref{lem:ground-state} shows the stability of $u_j$. On $B_1$, it is also straightforward to check that both $u_j$ and $Du_j$ are uniformly bounded. The only nonzero second derivative of $u_j$ is
$\partial_{11}u_j=-j\rho(jx_1)/4$. For $j\ge4$, we can compute that
\[
 \int_{B_{1/2}}|D^2u_j|^p
 \geq c(n,p)j^{p-1}\int_{-1}^1\rho(s)^pds
 \quad(\forall\;1<p<\infty),
\]
and the $L^\infty$ norm is $j\|\rho\|_{L^\infty(-1,1)}/4$. Therefore, $u_j,f_j$ satisfy our requirements.
\end{proof}

\section{BMO and Morrey estimates}\label{sec4}
In this section we provide the proof of Theorem \ref{thm:morrey}.
\begin{proof}[Proof of Theorem~\ref{thm:morrey}]
Set $G=B_{7/8}$ and $S=B_{7/8}\setminus\ol{B_{3/4}}$. By \eqref{eq:annulus-data} and \eqref{eq:critical-morrey}, we have
\begin{equation}\label{eq:morrey-energy-data}
 \|Du\|_{L^2(B_{15/16})}+\|Du\|_{M^{2,\beta_n}(G)}
 \le C(n)\mathfrak D(u),
\end{equation}
which implies that
\begin{equation}\label{eq:morrey-all-beta}
 \|Du\|_{M^{2,\beta}(G)}\le C(n,\beta)\mathfrak D(u)
 \qquad(\forall\;\beta_n\le\beta\le n).
\end{equation}
Since $S\subseteq B_{1}\setminus B_{3/4}$ has fixed positive measure, $|u_S|\le C(n)\mathfrak D(u)$. Besides, Poincar\'e's inequality on $G$ shows
\[
 \|u-u_S\|_{L^2(G)}\le C(n)\|Du\|_{L^2(G)}.
\]
Therefore, we get
\begin{equation}\label{eq:mean-control}
 \|u\|_{L^1(G)}+|u_S|\le C(n)\mathfrak D(u).
\end{equation}

For $2\le n\le10$, we use \eqref{eq:morrey-all-beta} with $\beta=2$. If $x\in B_{1/2}$ and $D:=B_{1/2}\cap B_r(x)$, then $D$ is convex,
$|D|\ge c(n)\min\{r,1\}^n$, and $\operatorname{diam}D\le2\min\{r,1\}$. Poincar\'e's inequality on convex domains (see \cite[Lem.\,7.16]{GilbargTrudinger2001} for example) and H\"older's inequality then imply
\[
 \fint_D|u-u_D|\le C(n)\min\{r,1\}|D|^{-1/2}\|Du\|_{L^2(D)}
 \le C(n)\mathfrak D(u).
\]
Together with \eqref{eq:mean-control}, this gives the BMO bound.

For part~(iii), put $\beta=2p/(p-2)$. The ranges of $p$ in part~(iii) give
$2<\beta\le n$ when $n=10$, and $\beta_n\le\beta\le n$ when $n\ge11$. Poincar\'e's inequality on convex domains (see \cite[Lem.\,7.16]{GilbargTrudinger2001} for example) gives
\[
 |u(x)-u_S|\le\frac{(\operatorname{diam}G)^n}{n|S|}
          \int_G|x-y|^{1-n}|Du(y)|\,dy\leq C(n)I_1(|Du|\mathbf1_G)\qquad(x\in G).
\]
From \cite[(3.3) in Thm.\,3.1]{Adams1975} we have the Riesz-potential embedding
\[
 I_1:M^{2,\beta}(\RR^n)\longrightarrow
       M^{2\beta/(\beta-2),\beta}(\RR^n),\qquad 2<\beta\le n.
\]
Therefore, from \eqref{eq:morrey-all-beta} we obtain
\[
 \|u-u_S\|_{M^{p,\beta}(B_{1/2})}\le C(n,p)\mathfrak D(u).
\]
Since the constant function $u_S$ has $M^{p,\beta}(B_{1/2})$ norm at most $C(n,p)|u_S|$, from \eqref{eq:mean-control} we derive the desired estimate.
\end{proof}

\section{Approximation and stable weak solutions}\label{sec5}

In this section, we first establish the interior approximation result (Theorem~\ref{thm:approximation}) and then use it to derive regularity estimates for stable weak solutions (Theorem~\ref{thm:weak}).

\subsection{Interior approximation}

\begin{proof}[Proof of Theorem~\ref{thm:approximation}]
The whole proof is divided into four steps.

\vskip0.1in

\emph{Step 1.}
Fix a nonzero function $\eta\in C_c^1(B_1)$ and set $C_0=1+\int_{B_1}|D\eta|^2/\int_{B_1}\eta^2$. Since $f$ is convex, $f'_-$ is nondecreasing. By stability,
\[
 \inf_{t\in\RR}f'_-(t)\int_{B_1}\eta^2
 \le \int_{B_1}f'_-(u)\eta^2
 \le \int_{B_1}|D\eta|^2.
\]
Hence we can choose $t_0<0$ such that $A:=f'_-(t_0)<C_0$, and put $K=At_0-f(t_0)$. The convexity of $f$ gives
\begin{equation}\label{eq:approx-linear-lower}
 f(t)\ge At-K\qquad(t\in\RR).
\end{equation}
Choose $r_0=r_0(n)\in(0,1/4)$ small such that
\begin{equation}\label{eq:approx-eigenvalue}
 \lambda_1(B_{2r_0})>C_0>A,
\end{equation}
where $\lambda_1(B)$ denotes the first Dirichlet eigenvalue of $-\Delta$ on $B$.

On $B_{2r_0}$, \eqref{eq:approx-eigenvalue} makes the bilinear form
\[
 (w,\phi)\longmapsto \int_{B_{2r_0}}Dw\cdot D\phi-A\int_{B_{2r_0}}w\phi
\]
coercive on $H_0^1(B_{2r_0})$. Since $u\in H_{\loc}^{1}(B_1)$ and
$-\Delta u=f(u)$ weakly, the functional
\[
 \phi\longmapsto A\int_{B_{2r_0}}u\phi-K\int_{B_{2r_0}}\phi
                 -\int_{B_{2r_0}}Du\cdot D\phi
\]
belongs to $H^{-1}(B_{2r_0})$.  The Lax--Milgram theorem therefore gives a unique function
$w\in H_0^1(B_{2r_0})$ satisfying
\[
 -\Delta w-Aw=Au-K-f(u)
\]
in the weak sense.  Set $V=u+w$.  Then
\begin{equation}\label{eq:approx-linear-V}
 -\Delta V=AV-K\quad\hbox{in }B_{2r_0},
 \qquad V-u\in H_0^1(B_{2r_0}).
\end{equation}
By \eqref{eq:approx-linear-lower}, we have $-\Delta w-Aw\le0$. Testing this inequality with $w^+$ and applying Poincar\'e's inequality, we obtain
\[
 \lambda_1(B_{2r_0})\int_{B_{2r_0}}|w^+|^2
 \le \int_{B_{2r_0}}|Dw^+|^2
 \le A\int_{B_{2r_0}}|w^+|^2.
\]
Since $\lambda_1(B_{2r_0})>A$ by \eqref{eq:approx-eigenvalue}, this implies $w^+=0$, and hence $V\le u$ almost everywhere.  The equation in
\eqref{eq:approx-linear-V} and interior elliptic regularity show that $V$ is smooth in
$B_{2r_0}$.  Consequently, we have
\begin{equation}\label{eq:approx-negative-u}
 \|u^-\|_{L^\infty(B_{r_0})}<\infty,
\end{equation}
since $u^-\le V^-$ a.e. in $B_{r_0}$.

The same construction on $B_{r_0}$ gives a unique function
$v\in H^1(B_{r_0})$ such that
\begin{equation}\label{eq:approx-v}
 -\Delta v=Av-K\quad\hbox{in }B_{r_0},
 \qquad v-u\in H_0^1(B_{r_0}),
\end{equation}
and $v\le u$ a.e. Let
$M_0=\|u^-\|_{L^\infty(B_{r_0})}$ and let $v_0$ solve the equation in
\eqref{eq:approx-v} with the constant boundary value $-M_0$.  Since
$u\ge-M_0$ a.e. and $\lambda_1(B_{r_0})>A$, the maximum principle gives
$v\ge v_0$.  Hence
\begin{equation}\label{eq:approx-v-lower}
 M:=\|v^-\|_{L^\infty(B_{r_0})}<\infty,
 \qquad -M\le v\le u\quad\hbox{a.e. in }B_{r_0}.
\end{equation}

If $\operatorname*{ess\,sup}_{B_{r_0}} f'_-(u)<\infty$, then
\eqref{eq:approx-negative-u} and convexity imply
$|f(u)|\le C(1+|u|)$ in $B_{r_0}$.  Interior elliptic regularity theory then guarantees that $u$ is classical in $B_{r_0}$.
In this case, we just take $f_k=f$ and $u_k=u$. It remains to consider the case
\begin{equation}\label{eq:approx-unbounded-slope}
 \operatorname*{ess\,sup}_{B_{r_0}} f'_-(u)=+\infty.
\end{equation}

\vskip0.1in

\emph{Step 2.}
For $k\ge1$, define
\begin{equation}\label{eq:approx-fk}
 f_k(t)=
 \begin{cases}
 f(t),&t\le k,\\
 f(k)+f'_-(k)(t-k),&t>k.
 \end{cases}
\end{equation}
The convexity of $f$ gives
\begin{equation}\label{eq:approx-fk-order}
 f_k\le f_{k+1}\le f,
 \qquad f_k(t)\uparrow f(t)\quad(t\in\RR),
 \qquad f_k(t)\ge At-K.
\end{equation}
Moreover, each $f_k$ is convex and is Lipschitz on $[-M,\infty)$, and we can take $\gamma\ge\max\{0,-f'_-(-M)\}$ such that $t\mapsto f_k(t)+\gamma t$ is nondecreasing on $[-M,\infty)$ for every $k$.

Set $u_k^{(0)}=v$.  Inductively, let $u_k^{(j)}\in H^1(B_{r_0})$ be the unique weak solution of
\begin{equation}\label{eq:approx-iteration}
 \begin{cases}
 -\Delta u_k^{(j)}+\gamma u_k^{(j)}
 =f_k(u_k^{(j-1)})+\gamma u_k^{(j-1)}&\text{in }B_{r_0},\\
 u_k^{(j)}-u\in H_0^1(B_{r_0}).
 \end{cases}
\end{equation}
Using \eqref{eq:approx-v}, \eqref{eq:approx-fk-order}, the monotonicity of
$f_k+\gamma\,\mathrm{id}$, and the weak maximum principle, we can inductively deduce that
\begin{equation}\label{eq:approx-order-jk}
 v\le u_k^{(j)}\le u_k^{(j+1)}\le u,
 \qquad
 u_k^{(j)}\le u_{k+1}^{(j)}
 \quad\text{a.e. in }B_{r_0}.
\end{equation}
Note that \eqref{eq:approx-order-jk} gives a uniform $L^2$ bound for $\{u_k^{(j)}\}_{j,k}$. For the gradient bound, testing
\eqref{eq:approx-iteration} with $u-u_k^{(j)}\in H_0^1(B_{r_0})$ yields
\begin{align*}
 \int_{B_{r_0}}|Du_k^{(j)}|^2
 &=\int_{B_{r_0}} Du_k^{(j)}\cdot Du
   -\int_{B_{r_0}} f_k(u_k^{(j-1)})(u-u_k^{(j)})\\
 &\quad
   +\gamma\int_{B_{r_0}} (u_k^{(j)}-u_k^{(j-1)})(u-u_k^{(j)}).
\end{align*}
By \eqref{eq:approx-fk-order} and \eqref{eq:approx-order-jk}, the last two terms are bounded above by
\[
 \int_{B_{r_0}} (K-Au_k^{(j-1)})(u-u_k^{(j)})
 +\gamma\int_{B_{r_0}} (u-v)^2.
\]
Since $|u_k^{(j-1)}|\le |u|+|v|$, Young's inequality gives
\begin{equation}\label{eq:approx-H1-uniform}
 \sup_{j,k}\|u_k^{(j)}\|_{H^1(B_{r_0})}<\infty.
\end{equation}

For fixed $k$, from \eqref{eq:approx-order-jk} and \eqref{eq:approx-H1-uniform} we can find
$u_k\in H^1(B_{r_0})$ such that
\[
 u_k^{(j)}\rightharpoonup u_k\quad\hbox{in }H^1(B_{r_0}),
 \qquad
 u_k^{(j)}\to u_k\quad\hbox{in }L^2(B_{r_0})\text{ and a.e. in }B_{r_0}
\]
as $j\to\infty$.  Since $f_k$ is Lipschitz on $[-M,\infty)$,
$f_k(u_k^{(j)})\to f_k(u_k)$ in $L^2(B_{r_0})$.  Passing to the limit in
\eqref{eq:approx-iteration} gives
\begin{equation}\label{eq:approx-uk-equation}
 -\Delta u_k=f_k(u_k)\quad\hbox{in }B_{r_0},
 \qquad u_k-u\in H_0^1(B_{r_0}),
\end{equation}
and
\begin{equation}\label{eq:approx-uk-order}
 v\le u_k\le u_{k+1}\le u\quad\hbox{a.e. in }B_{r_0}.
\end{equation}
Since $f_k$ is Lipschitz on $[-M,\infty)$, interior elliptic regularity theory yields that $u_k\in C^2(B_{r_0})$ is a classical solution.

For the stability of $u_k$, the convexity of $f$ and \eqref{eq:approx-fk} give
\[
 (f_k)'_-(t)\le f'_-(t)\qquad(\forall\;t\in\RR).
\]
Together with $u_k\le u$ and the monotonicity of $f'_-$, this gives
\[
 (f_k)'_-(u_k)\le f'_-(u_k)\le f'_-(u)\quad\hbox{a.e.}
\]
By the stability of $u$, it follows that
\[
 \int_{B_{r_0}}(f_k)'_-(u_k)\xi^2
 \le\int_{B_{r_0}}|D\xi|^2
 \qquad\bigl(\forall\;\xi\in C_c^1(B_{r_0})\bigr).
\]
This shows that $u_k$ is also stable.

\vskip0.1in

\emph{Step 3.}
By \eqref{eq:approx-H1-uniform} and \eqref{eq:approx-uk-order}, the monotone limit
$u_*\in H^1(B_{r_0})$ satisfies
\begin{equation}\label{eq:approx-limit}
 u_k\rightharpoonup u_*\quad\hbox{in }H^1(B_{r_0}),
 \qquad
 u_k\to u_*\quad\hbox{in }L^2(B_{r_0})\text{ and a.e. in }B_{r_0},
 \qquad
 v\le u_*\le u,
\end{equation}
$u_*-u\in H_0^1(B_{r_0})$, and
$f_k(u_k)\to f(u_*)$ a.e. To pass to the limit in the equation, note that by
\eqref{eq:approx-v-lower} and \eqref{eq:approx-uk-order},
$-M\le u_k\le u$.  Since $f$ is convex,
\[
 f_k(u_k)\le f(u_k)\le\max\{f(-M),f(u)\},
\]
whereas \eqref{eq:approx-fk-order} gives
$f_k(u_k)\ge Au_k-K$.  Consequently
\begin{equation}\label{eq:approx-domination}
 |f_k(u_k)|\le C\bigl(1+|u|+|f(u)|\bigr)
 \qquad\hbox{a.e. in }B_{r_0},
\end{equation}
where the constant $C$ is independent of $k$. The Lebesgue 
dominated convergence theorem then gives $f_k(u_k)\to f(u_*)$ in $L^1(B_{r_0})$. Passing to the limit in \eqref{eq:approx-uk-equation} yields
\begin{equation}\label{eq:approx-limit-equation}
 -\Delta u_*=f(u_*)\quad\hbox{weakly in }B_{r_0}.
\end{equation}

To identify $u_*=u$, we set $v=u-u_*\geq 0$. From the convexity of $f$ we get that
\begin{align}
    -\Delta v=f(u)-f(u_*)\geq f_-'(u_*)v,\nonumber
\end{align}
which implies
\begin{align}
    -\Delta v+(f_-'(u_*))^-v\geq 0.\nonumber
\end{align}
Since $u_*\geq -M$, we know that $(f_-'(u_*))^-$ is bounded. From \cite[Thm.\,8.19]{GilbargTrudinger2001} we have that either $v\equiv0$, or $v>0$ a.e. in $B_{r_0}$. In the first case we are done. In the second case, we can follow the argument in \cite[Proof of Thm.\,1.3, Step~5]{PengZhangZhou2026} to obtain a contradiction. Thus $u_*=u$, and \eqref{eq:approx-uk-order}--\eqref{eq:approx-limit} give $u_k\uparrow u$ a.e. and $u_k\to u$ in $L^2(B_{r_0})$.

\vskip0.1in

\emph{Step 4.} It remains to show the strong $H^1$ convergence $u_k\rightarrow u$. For this, we just need to argue as in \cite[Proof of Thm.\,1.3, Step~6]{PengZhangZhou2026}. Specifically, testing \eqref{eq:approx-uk-equation} with $u-u_k$ and using $f_k(u_k)\ge Au_k-K$ give
\[
 \limsup_{k\to\infty}\int_{B_{r_0}}|Du_k|^2\le\int_{B_{r_0}}|Du|^2.
\]
Weak lower semicontinuity gives the reverse inequality for the liminf, and thus $Du_k\to Du$ in $L^2(B_{r_0})$. Together with the $L^2$ convergence, this gives $u_k\to u$ strongly in $H^1(B_{r_0})$.
\end{proof}

\subsection{Regularity estimates for stable weak solutions}

\begin{proof}[Proof of Theorem~\ref{thm:weak}]
Let $(u_k,f_k)$ be the approximating sequence given by Theorem~\ref{thm:approximation}. Applying the scaled forms of Theorems~\ref{thm:gradient}, \ref{thm:holder}, and~\ref{thm:morrey} gives the corresponding estimates on $B_{r_0/2}$. For $1\le q\le2$, the right-hand sides converge by strong $H^1$ convergence and dominated convergence, using $f_k^-(u_k)\le C(1+|u|)$ with $C$ independent of $k$.

Moreover, \eqref{eq:geometric} and~\eqref{eq:matrix}, together with the strong $L^1$ convergence of $f_k(u_k)$, imply that $D^2u_k$ is uniformly integrable on $B_{r_0/2}$. By the Dunford--Pettis theorem, after passing to a subsequence, $D^2u_k$ converges weakly in $L^1(B_{r_0/2})$. Since $u_k\to u$ strongly in $H^1(B_{r_0})$, the limit is $D^2u$. Weak lower semicontinuity then yields the Hessian estimate. The remaining estimates pass to the limit by Fatou's lemma, strong $L^1$ convergence, and the Arzel\`a--Ascoli theorem.

The cases $q>2$, including $q=\infty$, follow from the case $q=2$ by H\"older's inequality. Finally, scaling and covering give the required estimates on $B_{1/2}$, with the full-ball data in Theorem~\ref{thm:weak} replacing the annular data.
\end{proof}

\section{Liouville-type theorems}\label{sec6}

This section is devoted to Liouville-type results for entire stable solutions. In the first subsection, we establish general growth criteria that play an important role in the proof of Theorem~\ref{thm:liouville} and are also of independent interest. In the second subsection, we prove Theorem~\ref{thm:liouville} using these criteria and the regularity estimates obtained in the preceding sections.

\subsection{Growth criteria for entire stable solutions}

Our first result shows that $L^1$ growth bounds for $u^-$ and $(\Delta u)^+$ over annuli yield corresponding bounds for $u$ and $(\Delta u)^+$ over balls, as well as $L^2$ estimates for $Du$ when $u$ is stable.

\begin{lemma}\label{lem:mass-conversion}
Let $u\in C^2(\RR^n)$, and set $A_R=B_{2R}\setminus\ol{B_R}$.
\begin{enumerate}
\item[(i)] Let $n\ge2$ and $\gamma>0$. If
\[
 \|u^-\|_{L^1(A_R)}=o(R^{n+\gamma}),\qquad
 \|(\Delta u)^+\|_{L^1(A_R)}=o(R^{n-2+\gamma}),
\]
then $\|u\|_{L^1(B_R)}=o(R^{n+\gamma})$ and
$\|(\Delta u)^+\|_{L^1(B_R)}=o(R^{n-2+\gamma})$. If $u$ is also a stable solution of \eqref{eq:pde}, then
\begin{equation}\label{eq:mass-energy}
 \int_{B_R}|Du|^2=o(R^{n-2+2\gamma}).
\end{equation}
All these conclusions also hold with $O$ in place of $o$ in both hypotheses and conclusions.
\item[(ii)] Let $n\ge3$. If
\[
 \|u^-\|_{L^1(A_R)}=o(R^n\log R),\qquad
 \|(\Delta u)^+\|_{L^1(A_R)}=o(R^{n-2}),
\]
then
\begin{equation}\label{eq:log-mass}
 \|u\|_{L^1(B_R)}=o(R^n\log R),\qquad
 \int_{B_R\setminus B_1}(\Delta u)^+(x)|x|^{2-n}\,dx=o(\log R).
\end{equation}
\end{enumerate}
\end{lemma}
\begin{proof}
We need to use the elementary fact, which can be easily proved by dyadic summation: for any locally integrable function $H\geq 0$ and any constant $d>0$, if
\begin{align}
    \int_{A_R}H=o(R^d),\nonumber
\end{align}
then
\begin{align}
    \int_{B_R}H=o(R^d).\nonumber
\end{align}
The same argument applies to $O(R^d)$ and $o(R^d\log R)$.

In part~(i), as both $n+\gamma$ and $n-2+\gamma$ are positive, the above fact gives the ball bounds for $u^-$ and $(\Delta u)^+$. Set
\[
 m(r)=\fint_{\partial B_r}u,\qquad H(r)=\int_{B_r}(\Delta u)^+.
\]
The divergence theorem gives
\[
 m'(r)=\frac{r^{1-n}}{|\mathbb S^{n-1}|}\int_{B_r}\Delta u
         \le\frac{r^{1-n}}{|\mathbb S^{n-1}|}H(r).
\]
Since $H(r)=o(r^{n-2+\gamma})$, for every $\varepsilon>0$, integration from a sufficiently large fixed radius gives
$m(r)\le C(u,\varepsilon)+C(n,\gamma)\varepsilon r^\gamma$.
Integrating in $r$, using $|u|=u+2u^-$, and letting $\varepsilon\downarrow0$ give the desired bound for $\|u\|_{L^1(B_R)}$. Under the $O$ assumptions, the same integration gives an $O(R^{n+\gamma})$ bound.

When $u$ is stable, $(\Delta u)^+=f^-(u)$. Rescaling Lemma~\ref{lem:endpoint} for $v(x)=u(2Rx)$, with nonlinearity $(2R)^2 f$, gives
\begin{equation*}
 \|Du\|_{L^2(B_R)}\le C(n)
 \left(R^{-n/2-1}\|u\|_{L^1(B_{2R})}
       +R^{1-n/2}\|(\Delta u)^+\|_{L^1(B_{2R})}\right).
\end{equation*}
The ball bounds of $u$ and $(\Delta u)^+$ then yield \eqref{eq:mass-energy}, also with $O$ in place of $o$.

In part~(ii), the above fact gives
$\|u^-\|_{L^1(B_R)}=o(R^n\log R)$ and $H(R)=o(R^{n-2})$. As in part (i), we can show that for every $\varepsilon>0$,
$m(r)\le C(u,\varepsilon)+C(n)\varepsilon\log r$ for all large $r$. Integration and $|u|=u+2u^-$ then give the first conclusion in \eqref{eq:log-mass}. On $A_{2^j}$,
\[
 \int_{A_{2^j}}(\Delta u)^+(x)|x|^{2-n}\,dx
 \le 2^{j(2-n)}\int_{A_{2^j}}(\Delta u)^+=:a_j,
 \qquad a_j\longrightarrow0.
\]
For $2^N\le R<2^{N+1}$, the second integral in \eqref{eq:log-mass} is bounded by
$\sum_{j=0}^Na_j=o(N)=o(\log R)$. The proof is complete.
\end{proof}

Our next two results concern the case $n=2$ and the case $n\geq 3$, respectively. 

\begin{lemma}\label{lem:planar}
Let $u\in C^2(\RR^2)$ be a stable solution of \eqref{eq:pde}. If
\[
 \int_{B_R}|Du|^2=O(R^2),
\]
then $u(x)=g(e\cdot x)$ for some $e\in\mathbb S^1$ and $g\in C^2(\RR)$.
\end{lemma}
\begin{proof}
For an integer $N\ge1$ and a large constant $L>0$, take a radial Lipschitz cutoff function $\eta_N$, equal to one on $B_L$, zero outside $B_{2^NL}$, and satisfying $|D\eta_N|\le C/(N2^jL)$ on the annulus
$B_{2^{j+1}L}\setminus B_{2^jL}$, $0\le j<N$. By \eqref{eq:geometric} we get
\[
 \int_{B_L}\mathcal A_u^2
 \le\frac C{N^2}\sum_{j=0}^{N-1}(2^jL)^{-2}
                          \int_{B_{2^{j+1}L}}|Du|^2
 \le\frac C N.
\]
By first letting $N\to\infty$ and then letting $L\to\infty$, we see that $\mathcal A_u=0$ almost everywhere in $\RR^2$. On the regular set $\{Du\ne0\}$, the level-set decomposition of the Hessian gives
\[
 \mathcal A_u^2=|D_T|Du||^2+|Du|^2|\mathrm{II}|^2,
\]
where $D_T$ is the tangential gradient and $\mathrm{II}$ is the second fundamental form of the regular level set. Thus $|Du|$ is constant along each connected regular level component and all its principal curvatures vanish. From \cite[Lem.\,2.11]{FarinaSciunziValdinoci2008} we obtain that
$u(x)=g(e\cdot x)$ for some $e\in\mathbb S^1$ and $g\in C^2(\RR)$.
\end{proof}

\begin{lemma}\label{lem:shell}
Let $n\ge3$, and let $u\in C^2(\RR^n)$ be a stable solution of \eqref{eq:pde}. Put $h=f^-(u)$ and $E(r)=\int_{B_r}|Du|^2$. If $3\le n\le10$, then, for $r>0$ and $\rho\ge9r/8$,
\begin{equation}\label{eq:shell-energy}
 \left(r^{2-n}E(r)\right)^{1/2}
 \le C(n)\left(\int_{B_{2\rho}\setminus B_\rho}|Du||x|^{1-n}\,dx
                  +\rho^{2-n}\|h\|_{L^1(B_{2\rho}\setminus B_\rho)}\right).
\end{equation}
For every $n\ge3$ and $0<r<R$,
\begin{equation}\label{eq:shell-splitting}
 \begin{aligned}
 \int_{B_R\setminus B_r}|Du||x|^{1-n}\,dx
 \le C(n)\biggl(&\fint_{B_{r/2}\setminus B_{r/4}}|u|
       +\fint_{B_{4R}\setminus B_{2R}}|u|\\
       &+\int_{B_{4R}\setminus B_{r/4}}h(x)|x|^{2-n}\,dx\biggr).
 \end{aligned}
\end{equation}
\end{lemma}
\begin{proof}
Applying \eqref{eq:radial-critical} with $R=16\rho/9$ gives
\[
 \left(r^{2-n}E(r)\right)^{1/2}
 \le C(n)\rho^{1-n/2}\|Du\|_{L^2(B_{16\rho/9}\setminus B_{4\rho/3})}.
\]
Combining \eqref{eq:endpoint-estimate} with a scaling and covering argument gives
\[
 \begin{aligned}
 \|Du\|_{L^2(B_{16\rho/9}\setminus B_{4\rho/3})}
 \le C(n)\bigl(&\rho^{-n/2}\|Du\|_{L^1(B_{2\rho}\setminus B_\rho)}\\
               &+\rho^{1-n/2}\|h\|_{L^1(B_{2\rho}\setminus B_\rho)}\bigr).
 \end{aligned}
\]
\eqref{eq:shell-energy} follows from the above two estimates.

For \eqref{eq:shell-splitting}, we just need to modify the proof of \cite[Prop.\,1.9]{PengZhangZhou2024}, which treats the case $f\geq 0$ and gives \eqref{eq:shell-splitting} with $h=0$. Specifically, we replace the identity $|\Delta u|=-\Delta u$ used there by
\[
 |\Delta u|=-\Delta u+2h.
\]
It is not hard to see that the contribution of $2h$ is bounded by
\[
 C(n)\int_{B_{4R}\setminus B_{r/4}}h(x)|x|^{2-n}\,dx.
\]
This gives \eqref{eq:shell-splitting} for every $n\ge3$.
\end{proof}

\subsection{Liouville theorems}
\begin{proof}[Proof of Theorem~\ref{thm:liouville}]
For simplicity, we set $h=f^-(u)=(\Delta u)^+$.

\emph{Part (i).} Lemma~\ref{lem:mass-conversion}(i), with $\gamma=1$, gives $\int_{B_R}|Du|^2=o(R^2)$. Lemma~\ref{lem:planar} then implies $u(x)=g(e\cdot x)$ for some $e\in\mathbb S^1$ and $g\in C^2(\RR)$. Suppose $h_0(t):=f^-(g(t))$ is positive at some point. By continuity it is bounded below by a positive number on a fixed interval $I$. For all sufficiently large $R$, we have
\[
 \int_{A_R}h_0(e\cdot x)\,dx\ge cR\int_I h_0(t)\,dt,
\]
contrary to the  hypothesis. Therefore, we have $f(g)\ge0$ and $g''=-f(g)\le0$. If $g$ is nonconstant, the concavity of $g$ shows $g(t)\le-c|t|+C$ at one end of $\RR$. Therefore, for $R$ large, $u^-\ge c'R$ on $A_R\cap \{e\cdot x<-R\}$ or $A_R\cap \{e\cdot x>R\}$. However, this contradicts the assumption $\|u^-\|_{L^1(A_R)}=o(R^3)$. Hence $u$ is constant.

\vskip0.1in

\emph{Part (ii).} Lemma~\ref{lem:mass-conversion}(i), with $\gamma=\alpha_n$, gives
\[
 \int_{B_R}|Du|^2=o(R^{n-2+2\alpha_n}).
\]
For fixed $r>0$, from \eqref{eq:energy-decay} we have
\[
 \int_{B_r}|Du|^2
 \le C(n,r)R^{-(n-2+2\alpha_n)}\int_{B_R}|Du|^2=o(1).
\]
First letting $R\to\infty$ and then letting $r\to\infty$ proves that $Du\equiv0$, which means that $u$ is constant.

\vskip0.1in

\emph{Part (iii).} By Lemma~\ref{lem:mass-conversion}(ii), we have
\begin{align}\label{add13}
    \|u\|_{L^1(B_T)}=o(T^n\log T),\qquad
 \int_{B_T\setminus B_1}h(x)|x|^{2-n}\,dx=o(\log T).
\end{align}
Define $\omega(R)=\sup_{\rho\ge R}\rho^{2-n}\|h\|_{L^1(A_\rho)}$, which tends to zero as $R\to\infty$. Fix $r>4$, take $R$ large, and set $N=\lfloor\log_2R\rfloor$. Averaging \eqref{eq:shell-energy} over $\rho=2^jR$, $0\le j<N$, yields
\[
 \left(r^{2-n}\int_{B_r}|Du|^2\right)^{1/2}
 \le\frac{C(n)}N\int_{B_{R^2}\setminus B_R}|Du||x|^{1-n}\,dx
       +C(n)\omega(R).
\]
Next we apply \eqref{eq:shell-splitting} with inner radius $4$ and outer radius $R^2$ to derive
\begin{align}
    \int_{B_{R^2}\setminus B_R}|Du||x|^{1-n}\,dx
 \le C(n)\biggl(&\fint_{B_{2}\setminus B_{1}}|u|
       +\fint_{B_{4R^2}\setminus B_{2R^2}}|u|\nonumber\\
       &+\int_{B_{4R^2}\setminus B_{1}}h(x)|x|^{2-n}\,dx\biggr).\nonumber
\end{align}
From \eqref{add13} we see that the right-hand side is $o(\log R)$. Thus, first letting $R\to\infty$ and then letting $r\to\infty$ yields $Du\equiv 0$, which shows that $u$ is constant.

\vskip0.1in

\emph{Part (iv).} Since $\kappa_n=(2-\beta_n)/2$, applying \eqref{eq:annulus-data} in a scaled form gives
\[
 \|Du\|_{L^2(B_R)}\le C(n)
 \left(R^{-n/2-1}\|u\|_{L^1(A_R)}
        +R^{1-n/2}\|h\|_{L^1(A_R)}\right)
 =o(R^{(n-\beta_n)/2}).
\]
Hence $\int_{B_R}|Du|^2=o(R^{n-\beta_n})$. For  $0<r<R/2$, from \eqref{eq:radial-high} we have
\[
 r^{\beta_n-n}\int_{B_r}|Du|^2
 \le C(n)R^{\beta_n-n}\int_{B_R\setminus B_{3R/4}}|Du|^2.
\]
First letting $R\to\infty$ and then letting $r\to\infty$ shows that $u$ is constant.
\end{proof}

\appendix
\section{A short proof of a radial Poincar\'e inequality}\label{app:radial-poincare}

Cabr\'e's radial Poincar\'e inequality bounds the oscillation of a superharmonic function in terms of its radial derivative; see \cite[Thm.\,1.4]{Cabre2022} and \cite[Prop.\,C]{Cabre2026}. It was originally introduced to provide a new, quantitative proof of the interior H\"older estimates in \cite{CabreFigalliRosOtonSerra2020}. The following result extends the annular version of this inequality to general functions, while also providing a short alternative proof in the superharmonic setting.

\begin{lemma}\label{lem:appendix-radial-poincare}
Let $n\ge2$ and $A=B_1\setminus\ol{B_{1/2}}$. For every $u\in W^{2,1}(A)$,
\begin{equation}\label{eq:appendix-radial-poincare}
 \|u-u_A\|_{L^1(A)}
 \le C(n)\left(\|u_r\|_{L^1(A)}+\|(\Delta u)^+\|_{L^1(A)}\right),
\end{equation}
where $u_r=x\cdot Du/|x|$ and $u_A=\fint_Au$.
\end{lemma}

\begin{proof}
By approximation, we may assume that $u$ is smooth on $A$. Choose a nonnegative function $\eta\in C_c^\infty((1/2,1))$ such that
\[
 \int_{1/2}^1\eta(s)s^{n-3}\,ds=1,
\]
and define, for $\omega\in\mathbb S^{n-1}$,
\[
 U(\omega)=\int_{1/2}^1\eta(s)s^{n-3}u(s\omega)\,ds.
\]
The polar decomposition of the Laplacian gives
\[
 \Delta_{\mathbb S^{n-1}}u(s\omega)
 =s^2\Delta u(s\omega)-s^2u_{ss}(s\omega)-(n-1)su_s(s\omega).
\]
Multiplying it by $\eta(s)s^{n-3}$ and integrating by parts in $s$ yields
\begin{equation}\label{eq:appendix-spherical-laplacian}
 \Delta_{\mathbb S^{n-1}}U(\omega)
 =\int_{1/2}^1\eta(s)s^{n-1}\Delta u(s\omega)\,ds
  +\int_{1/2}^1\eta'(s)s^{n-1}u_s(s\omega)\,ds.
\end{equation}
Taking positive parts and integrating on the sphere then gives
\begin{equation}\label{eq:appendix-positive-lap}
 \|(\Delta_{\mathbb S^{n-1}}U)^+\|_{L^1(\mathbb S^{n-1})}
 \le C(n)\left(\|(\Delta u)^+\|_{L^1(A)}+\|u_r\|_{L^1(A)}\right).
\end{equation}
Since the integral of a spherical Laplacian is zero, we have
$\|\Delta_{\mathbb S^{n-1}}U\|_{L^1}=2\|(\Delta_{\mathbb S^{n-1}}U)^+\|_{L^1}$. Moreover, from Green's representation in $\mathbb S^{n-1}$ we have
\begin{equation}\label{eq:appendix-sphere-poincare}
 \|U-U_{\mathbb S^{n-1}}\|_{L^1(\mathbb S^{n-1})}
 \le C(n)\|\Delta_{\mathbb S^{n-1}}U\|_{L^1(\mathbb S^{n-1})}.
\end{equation}

To compare $u(x)$ with $U(x/|x|)$, note that for $1/2<r<1$,
\[
 u(r\omega)-U(\omega)
 =\int_{1/2}^1\eta(s)s^{n-3}\bigl(u(r\omega)-u(s\omega)\bigr)\,ds,
\]
and
\[
 |u(r\omega)-u(s\omega)|
 \le\int_{\min\{r,s\}}^{\max\{r,s\}}|u_t(t\omega)|\,dt.
\]
Integrating first in $r$ and $s$ and then in $\omega$, by Fubini's theorem we derive that
\begin{equation}\label{eq:appendix-ray}
 \|u-U(x/|x|)\|_{L^1(A)}\le C(n)\|u_r\|_{L^1(A)}.
\end{equation} Combining \eqref{eq:appendix-positive-lap}--\eqref{eq:appendix-ray} then yields
\[
 \|u-U_{\mathbb S^{n-1}}\|_{L^1(A)}
 \le C(n)\left(\|u_r\|_{L^1(A)}+\|(\Delta u)^+\|_{L^1(A)}\right).
\]
As for every constant $c$,
\[
 |u_A-c|\le\fint_A|u-c|,
 \qquad
 \|u-u_A\|_{L^1(A)}\le2\|u-c\|_{L^1(A)},
\]
taking $c=U_{\mathbb S^{n-1}}$ gives \eqref{eq:appendix-radial-poincare} for smooth $u$.
\end{proof}

\begin{remark}
The radial Poincar\'e inequality above also provides an alternative approach to interior H\"older regularity. Indeed, following Cabr\'e's argument in \cite[Sec.\,2.6]{Cabre2026} and combining Lemma~\ref{lem:appendix-radial-poincare} with \eqref{eq:radial-budget} and the gradient estimate in Theorem~\ref{thm:gradient}, we can apply H\"older's inequality and a standard iteration argument to obtain, for $3\le n\le9$ and every $p>n/2$,
\[
 \|u\|_{C^{0,\alpha}(\ol{B_{1/2}})}
 \le C(n,p)\left(
 \|u\|_{L^1(B_1)}
 +\|f^-(u)\|_{L^p(B_1)}
 \right)
\]
for some $\alpha=\alpha(n,p)\in(0,1)$. We omit the details. Unlike Theorem~\ref{thm:holder}, this estimate requires higher integrability of $f^-(u)$.
\end{remark}

\section*{Acknowledgements}

	The authors are grateful to Professor Louis Dupaigne for carefully reading an earlier version of the manuscript and for his helpful comments and stimulating correspondence.

\section*{Declarations}

\noindent\textbf{Conflict of interest.}
The authors declare that they have no conflicts of interest.

\medskip

\noindent\textbf{Data availability.}
Data sharing is not applicable to this article, as no datasets were generated or analyzed during the current study.

\medskip

\noindent\textbf{Generative AI disclosure.}
The research leading to this paper began in March 2025. Before using AI-assisted tools, the authors had independently developed the overall research framework and established all the main results in their essential form, which they subsequently continued to refine and improve. More recently, ChatGPT was used to assist in locating relevant references, filling in technical details, checking calculations and proofs, and improving the exposition of the manuscript. The authors have carefully verified the mathematical arguments and take full responsibility for the results and the final content of this paper.

\bibliographystyle{amsplain}
\bibliography{references}

\end{document}